\documentclass[11pt,reqno,a4paper]{amsart} 

\usepackage[latin1]{inputenc}
\usepackage{amsmath}
\usepackage{amssymb}
\usepackage{euscript}
\usepackage{mathrsfs}
\usepackage{wasysym}
\usepackage[all]{xy}
\usepackage{graphicx}
\usepackage{color}
\usepackage{multirow}
\usepackage{pifont}
\usepackage[table]{xcolor}
\usepackage[colorlinks=true,linktocpage=true,pagebackref=false, urlcolor=black, citecolor=black,linkcolor=black]{hyperref}
\usepackage{cite}

\newtheorem{thm}{Theorem}

\newtheorem{lem}[thm]{Lemma}

\newtheorem{cor}[thm]{Corollary}

\newtheorem{assertion}[thm]{Assertion}

\theoremstyle{definition}
\newtheorem{defn}[thm]{Definition}

\newtheorem{remark}[thm]{Remark}
\newtheorem{examples}[thm]{Examples}
\newtheorem{example}[thm]{Example}

\theoremstyle{remark}

\newtheorem{quest}[thm]{Question}

 \newcommand{\N}{{\mathbb N}}
\newcommand{\Z}{{\mathbb Z}} \newcommand{\R}{{\mathbb R}}
 \newcommand{\C}{{\mathbb C}}

\newcommand{\sph}{{\mathbb S}}

 \newcommand{\I}{{\mathcal I}}

\newcommand{\Pp}{{\EuScript P}}

\newcommand{\Env}{{\EuScript E}}

\newcommand{\Ii}{{\EuScript I}}
\newcommand{\Qq}{{\EuScript Q}}

\newcommand{\partialIIk}{\partial_{\,\II,\kappa}}

\newcommand{\x}{{\tt x}}

\newcommand{\kk}{\mathrm{k}}

\newcommand{\II}{\mathbb{I}}
\newcommand{\hh}{\mathbb{H}}
\newcommand{\oo}{\mathbb{O}}
\newcommand{\mc}{\mathcal}
\newcommand{\mr}{\mathrm}
\newcommand{\mscr}{\mathscr}
\newcommand{\sss}{\scriptscriptstyle}
\newcommand{\n}{\mathtt{n}}
\newcommand{\m}{\mathtt{m}}
\newcommand{\DD}{\overline{\EuScript D}}
\newcommand{\zz}{\mathtt{z}}
\newcommand{\xx}{\mathtt{x}}
\newcommand{\bs}{{$\scriptscriptstyle\blacksquare$}}

\begin{document}
\title[]{Slice Fueter-regular functions}

\author{Riccardo Ghiloni}
\address{Dipartimento di Matematica, Via Sommarive, 14, Universit\`a di Trento, 38123 Povo-Trento (ITALY)}
\email{riccardo.ghiloni@unitn.it}
\thanks{The author is supported by GNSAGA of INDAM}

\begin{abstract}
Slice Fueter-regular functions, originally called slice Dirac-regular functions, are generalized holomorphic functions defined over the octonion algebra $\oo$, recently introduced by M. Jin, G. Ren and I. Sabadini.
A function $f:\Omega_D\subset\oo\to\oo$ is called (quaternionic) slice Fueter-regular if, given any quaternionic subalgebra $\hh_\II$ of $\oo$ generated by a pair $\II=(I,J)$ of orthogonal imaginary units $I$ and $J$ ($\hh_\II$ is a `quaternionic slice' of $\oo$), the restriction of $f$ to $\Omega_D\cap\hh_\II$ belongs to the kernel of the corresponding Cauchy-Riemann-Fueter operator $\frac{\partial}{\partial x_0}+I\frac{\partial}{\partial x_1}+J\frac{\partial}{\partial x_2}+(IJ)\frac{\partial}{\partial x_3}$.

The goal of this paper is to show that slice Fueter-regular functions are standard (complex) slice functions, whose stem functions satisfy a Vekua system 
having exactly the same form of the one characterizing axially monogenic functions of degree zero. The mentioned standard sliceness of slice Fueter-regular functions is able to reveal their `holomorphic nature': slice Fueter-regular functions have Cauchy integral formulas, Taylor and Laurent series expansions, and a version of Maximum Modulus Principle, and each of these properties is global in the sense that it is true on genuine $8$-dimesional domains of $\oo$. Slice Fueter-regular functions are real analytic. Furthermore, we introduce the global concepts of spherical Dirac operator $\Gamma$ and of slice Fueter operator $\overline{\vartheta}_F$ over octonions, which allow to characterize slice Fueter-regular functions as the $\mscr{C}^2$-functions in the kernel of $\overline{\vartheta}_F$ satisfying a second order differential system associated with $\Gamma$. 

The paper contains eight open problems.
\end{abstract}

 .

\maketitle 

\section{Introduction} \label{sec:1}

In \cite{JRS2019} M. Jin, G. Ren and I. Sabadini introduced and studied the class of slice Dirac-regular functions. Slice Dirac-regularity is a new concept of `holomorphy over octonions', which merges the classical quaternionic Fueter-regularity and the more recent slice regularity ideas.

In the present paper we use the term {\it slice Fueter-regular} function to indicate a slice Dirac-regular function in the sense of \cite{JRS2019}. See final Subsection \ref{subsec:sdrf} for the reason of this choice.

Our aim is to prove that slice Fueter-regular functions are standard slice fun\-ctions satisfying a Vekua-type system, whose form is identical to the one characterizing axially monogenic functions of degree zero. The mentioned standard sliceness has some basic consequences. Slice Fueter-regular functions $f:\Omega\subset\oo\to\oo$ satisfy Cauchy-type integral formulas of {\it global nature}, i.e. valid on the whole domain $\Omega$ of $f$. At each point $y$ of $\Omega$, $f$ admits `Fueter polynomial' series expansions of {\it global nature}, i.e. converging in standard $8$-dimensional Euclidean neighborhoods of $y$. In particular, slice Fueter-regular functions are real analytic. Furthermore, thanks to a new result which allows to pointwise `annihilate' the camshaft effect occuring in the slice product between any given slice function and certain pairs of octonionic constants, we are able to show that slice Fueter-regular functions satisfy a version of Maximum Modulus Principle. 

In the 1934 paper \cite{fueter1934}, R. Fueter introduced a notion of `holomorphic function over quaternions', making use of the generalized Cauchy-Riemann operator $\DD=\frac{\partial}{\partial x_0}+i\frac{\partial}{\partial x_1}+j\frac{\partial}{\partial x_2}+k\frac{\partial}{\partial x_3}$. Given a non-empty open subset $\Omega$ of the quaternion algebra $\hh$, a function $f:\Omega\to\hh$ is nowadays called {\it (left) Fueter-regular} if it is of class $\mscr{C}^1$ in the usual real sense, and
\[
\DD f(x)=\frac{\partial f}{\partial x_0}(x)+i\frac{\partial f}{\partial x_1}(x)+j\frac{\partial f}{\partial x_2}(x)+k\frac{\partial f}{\partial x_3}(x)=0 \quad \text{for each $x\in \Omega$},
\]
where $(x_0,x_1,x_2,x_3)$ are the real coordinates of $x$, i.e. $x=x_0+x_1i+x_2j+x_3k$. The operator $\DD$ is called Cauchy-Riemann-Fueter operator, even though it has been originally introduced by G. C. Moisil in \cite{Mo1931}. The theory of Fueter regularity is now well-established, including its generalizations over Clifford algebras (see e.g. \cite{BDS1982,DSS1992,GHS2008,Su1979}). It shares several fundamental analytic results with the theory of holomorphic functions of a complex variable; for instance, it has its own Cauchy integral formula, Taylor and Laurent series expansions, and Maximum Modulus Principle. However, the algebraic nature is lost in the sense that classical polynomials with quaternionic coefficients on the right are almost always not Fueter-regular. For instance, given the monomial $xa$ with $a\in\hh$, $\DD(xa)=-2a$ so $xa$ is Fueter-regular if and only if $a=0$.

In 2006 G. Gentili and D. C. Struppa \cite{GS2006,GS2007} introduced a new notion of regularity for quaternionic functions $f:\Omega\subset\hh\to\hh$, the one of {\it slice regular} function, starting from an idea of G. C. Cullen \cite{Cu1965}. These functions have remarkable properties. Besides Cauchy-type integral formulas, Taylor- and Laurent-type series expansions, and Maximum and Minimum Modulus Principles, they include all classical quaternionic polynomials $\sum_{h=0}^dx^ha_h$ and their multiplicative structure. Furthermore, the theory of slice regular functions has applications to functional calculus and mathe\-matical foundation of quantum mechanics (see e.g. \cite{CSS2011,GMP2013,GMP2017} and their references), classification of orthogo\-nal complex structures in $\R^4$ (see e.g. \cite{Alta2018,AS2019,GSaS2014}), and operator semigroup theory (see e.g. \cite{CS2011,GR2016,GR2018}). Slice regulari\-ty was extended to Clifford algebras and octonions in \cite{CSS2009,GS2008,GS2010}. Paper \cite{GP2011} contains an unified approach valid over all real alternative *-algebras, based on the notion of {\it stem} function. The latter notion allows to define the concept of {\it slice} function, which extends the one of slice regular function to the set-theoretic, not necessarily holomorphic, setting. It is important to remember that the germ of the idea of slice function was already contained in \cite{fueter1934,Ri1960,Cu1965} for associative algebras and in \cite{Sce1957,DS1973} for non-associative algebras. 


\subsection{Structure of the paper} In the remaining part of present introductory section we briefly recall basic concepts and results from the theories of slice functions and slice Fueter-regular functions over the octonion algebra $\oo$. We conclude the section fixing some notations we will use throughout the next two sections of the paper.

Section \ref{sec:results} is devoted to the presentation of our results. 

First, we compare the mentioned theories in Subsection \ref{subsec:comparison}. In Subsection \ref{subsec:sfo} we introduce and study the concepts of spherical Dirac operator $\Gamma$ and of slice Fueter operator $\overline{\vartheta}_F$ over $\oo$, which allow to describe slice Fueter-regular functions as the $\mscr{C}^2$-functions in the kernel of $\overline{\vartheta}_F$ satisfying a second order differential system associated with $\Gamma$. 

In Sections 5 and 6 of \cite{JRS2019}, the authors prove that, given any slice Fueter-regular function $f$ and fixed any pair $\II=(I,J)$ of orthogonal imaginary units $I$ and $J$ of $\oo$, the restriction of $f$ to the subalgebra $\hh_{\II}$ of $\oo$ generated by $I$ and $J$ satisfies a version of Cauchy integral formula, and it admits a Taylor or Laurent series expansion via Fueter-type polynomials or rational functions. All these results, including Theorems 5.4, 6.4 and 6.5, are {\it not global} in the sense that they give information on $f$ only over $\hh_\II$, see Remarks 6.2 and 6.3 of \cite{JRS2019}. In Subsections \ref{subsec:cauchy}, \ref{subsec:taylor-expansions} and \ref{subsec:laurent-expansions} below, we present global versions of the mentioned results. Subsection \ref{subsec:cauchy} contains also a global slice version of Borel-Pompeiu formula. In Subsection \ref{subsec:mmp}, we present a version of Maximum Modulus Principle for slice Fueter-regular functions. As we said, this Principle is based on a new result, Lemma~\ref{lem:F1-F2}, which is able to pointwise `annihilate' the camshaft effect occurring in our non-associative setting. 

The proofs of our results are postponed to Section \ref{sec:proofs}. 

Section \ref{sec:final} contains some final comments concerning possible generalizations of the concept of slice Fueter-regular function.

In the present paper we formulate eight open problems, see {\it Questions} \ref{quest:1}, \ref{quest:2}, \ref{quest:3}, \ref{quest:real-vector-basis}, \ref{quest:4}, \ref{quest:38}, \ref{quest:50} and~\ref{quest:51}. We conjecture that {\it Questions} \ref{quest:1}, \ref{quest:2}, \ref{quest:3} and \ref{quest:4} have substantially affirmative, possibly demanding, solutions.


\subsection{Slice functions over $\oo$}\label{subsec:GS} Over octonions the stem function approach to slice regular fun\-ctions is as follows. By the Cayley-Dickson process, the division algebra of octonions can be defined as $\oo:=\hh+\ell\hh$ with $(a+\ell b)(c+\ell d):=(ac-d\overline{b})+\ell(\overline{a}d+cb)$ and $\overline{a+\ell b}:=\overline{a}-\ell b$ for each $a,b,c,d\in\hh$, where $\overline{a}$ denotes the usual quaternionic conjugation of $a$. As a consequence, $(1,i,j,k, \ell, \ell i, \ell j,\ell k)$ can be identified with the canonical real vector basis of $\oo$, and $\R$, $\C$ and $\hh$ with the real subalgebras of $\oo$ generated by $\{1\}$, $\{i\}$ and $\{i,j\}$, respectively. We assume that $\oo$ is equipped with the Euclidean topology of $\R^8$, and $|x|$ is the usual Euclidean norm of $x$ in $\R^8$.

Let $\sph=\{I\in\oo\,:\,I^2=-1\}$ be the $6$-sphere of imaginary units of~$\oo$. For each $I\in\sph$, we indicate $\C_I$ the real subalgebra of $\oo$ generated by $I$, which coincides with its real vector subspace $\mr{Span}(1,I)$. The octonion algebra $\oo$ decomposes into `complex slices' as follows:
\[
\oo=\bigcup_{I\in\sph}\C_I, \quad \text{ where $\C_I\cap\C_J=\R\;$ if $I\neq\pm J$.}
\]

Let $D$ be a subset of $\C=\R^2$, invariant under the complex conjugation $(\alpha,\beta)\mapsto(\alpha,-\beta)$. Define the circularization $\Omega_D$ of $D$ in $\oo$ by setting
\[
\Omega_D:=\{\alpha+\beta I\in\oo\,:\,(\alpha,\beta)\in D,I\in\sph\}.
\]
Note that, if $D$ is open in $\C$, then $\Omega_D$ is open in $\oo$. A subset of $\oo$ is said to be circular if it is equal to some $\Omega_D$. Suppose that $D\neq\emptyset$. A function $F=(F_1,F_2):D\to\oo^2$ is said to be a {\it stem function} if $F$ is conjugation-intrinsic in the sense that $F(\overline{z})=\overline{F}(z)$ for each $z\in D$, where $\overline{F}(z):=(F_1(z),-F_2(z))$. Equivalently, $F$ is a stem function if and only if
\begin{equation}\label{eq:stem}
\text{$F_1(\alpha,-\beta)=F_1(\alpha,\beta)\;$ and $\;F_2(\alpha,-\beta)=-F_2(\alpha,\beta)\quad$ for each $(\alpha,\beta)\in D$.}
\end{equation}
We denote $\mr{Stem}(D,\oo^2)$ the set of all stem functions from $D$ to $\oo^2$. A function $f:\Omega_D\to\oo$ is called {\it (left) slice function} if there exists a stem function $F=(F_1,F_2):D\to\oo^2$ such that, if $x=\alpha+\beta I\in\Omega_D$ for some $I\in\sph$ and $z=(\alpha,\beta)\in D$, then
\[
f(\alpha+\beta I)=F_1(z)+IF_2(z).
\]
Note that, thanks to \eqref{eq:stem}, the above definition of $f$ is well-posed. Indeed, if $x\in\R$, then $\beta=0$ and $I$ can be choosen arbitrarily in $\sph$; anyway, $F_2(z)=0$ so $f(x)=F_1(z)$. If $x\not\in\R$, then there exist, and are unique, $(\alpha,\beta)\in D$ and $I\in\sph$ such that $\beta>0$ and $x=\alpha+\beta I=\alpha+(-\beta)(-I)$; anyway, $F_1(\alpha,\beta)+IF_2(\alpha,\beta)=F_1(\alpha,-\beta)+(-I)F_2(\alpha,-\beta)$. We say that $f$ is induced by $F$ and we write $f=\I(F)$. The slice function $f$ is induced by a unique stem function $F$. Indeed, given any $J\in\sph$, $F_1(\alpha,\beta)=\frac{1}{2}(f(\alpha+\beta J)+f(\alpha-\beta J))$ and $F_2(\alpha,\beta)=-\frac{1}{2}J(f(\alpha+\beta J)-f(\alpha-\beta J))$. We denote $\mc{S}(\Omega_D,\oo)$ the set of all slice functions from $\Omega_D$ to $\oo$. As we just know, the map $\I:\mr{Stem}(D,\oo^2)\to\mc{S}(\Omega_D,\oo)$ is a bijection. For details concerning the relations between the differential regularity of stem functions and the one of the induced slice functions, we refer the reader to \cite[Proposition 7]{GP2011}.

In general the pointwise product of two slice functions is not a slice function. Let us recall the notion of slice product. Let $f=\I(F),g=\I(G)\in\mc{S}(\Omega_D,\oo)$ with $F=(F_1,F_2)$ and $G=(G_1,G_2)$. Equip $\oo^2$ with the multiplication induced by the tensor product $\oo\otimes_{\R}\C$: $(x,y)(x',y'):=(xx'-yy',xy'+yx')$. Now the pointwise product $FG$ is again a stem function, so we can define the {\it slice product} of $f$ and $g$ as $f\cdot g:=\I(FG)$. If $f$ is slice preserving, i.e. $F_1$ and $F_2$ are real-valued, then $f\cdot g$ is equal to the pointwise product $fg$. For instance, for each $n,m\in\N$ and $a\in\oo$, the pointwise-defined monomial function $\oo\to\oo$, $x\mapsto x^n\overline{x}^ma=\overline{x}^mx^na$ is a slice function. The same is true for finite sums of such monomial functions. If $f$ and $g$ are $\C_I$-preserving for some $I\in\sph$, i.e. $f(\Omega_D\cap\C_I)\subset\C_I$ and $g(\Omega_D\cap\C_I)\subset\C_I$, then $f\cdot g=fg$ on $\Omega_D\cap\C_I$. Furthermore, by Artin's theorem \cite[Theorem 3.1]{Sch1966}, if $f$ is $\C_I$-preserving and $g$ is constantly equal to an octonion $c$, then $f\cdot g=fc$ on $\Omega_D\cap\C_I$. The slice product makes $\mc{S}(\Omega_D,\oo)$ a non-associative real alternative algebra, whose center contains all slice preserving functions.

Let $F^c:D\to\oo^2$ be the stem function $F^c(z):=(\overline{F_1(z)},\overline{F_2(z)})$ and let $f^c:=\I(F^c)$. The slice preserving function $N(f):\Omega_D\to\oo$ defined by $N(f):=f\cdot f^c$ is called {\it normal function} of $f$. This is a very important function. For instance, $N(f\cdot g)=N(f)N(g)$, and the zero set $V(N(f))$ of $N(f)$ is the smallest circular subset of $\oo$ containing the zero set $V(f)$ of $f$. Furthermore, if $V(N(f))\neq\Omega_D$ then the function $f^{-\sss\bullet}:\Omega_D\setminus V(N(f))\to\oo$ defined by $f^{-\sss\bullet}:=(N(f))^{-1}f^c$ is the unique slice function such that $f\cdot f^{-\sss\bullet}=f^{-\sss\bullet}\cdot f=1$ on $\Omega_D\setminus V(N(f))$. The function $f^{-\sss\bullet}$ is called {\it slice reciprocal}, or {\it slice multiplicative inverse}, of $f$. If $f$ is slice preserving then $f^{-\sss\bullet}(x)=(f(x))^{-1}$ for each $x\in\Omega_D\setminus V(N(f))$. If $f$ is $\C_I$-preserving then $f^{-\sss\bullet}(x)=(f(x))^{-1}$ for each $x\in(\Omega_D\setminus V(N(f)))\cap\C_I$.

Given $\xi\in\oo$, the real polynomial $\Delta_\xi(x)=x^2-2x\mr{Re}(\xi)+|\xi|^2$ is equal to the normal function $N(\xi-x)$ of $\xi-x$, so the zero set of $\Delta_\xi$ coincides with the set $\sph_\xi:=\{\mr{Re}(\xi)+I|\mr{Im}(\xi)|\in\oo\,:\,I\in\sph\}$, where $\mr{Re}(\xi):=\frac{1}{2}(\xi+\overline{\xi})\in\R$ and $\mr{Im}(\xi):=\frac{1}{2}(\xi-\overline{\xi})$. It turns out that $\sph_\xi$ is equal to the conjugation class of $\xi$ in $\oo$, i.e. $\sph_\xi=\{x\xi x^{-1}\in\oo\,:\,x\in\oo\setminus\{0\}\}$. For further details on slice product, normal function and slice reciprocal, see \cite{GP2011,GPS2017,GPS2019}. 

Suppose that $D$ is open in $\C$ and consider a stem function $F=(F_1,F_2):D\to\oo^2$ of class $\mscr{C}^1$ in the usual real sense. 
Define the functions $\partial F/\partial\bar{z}:D\to\oo^2$ and $F'_2:D\setminus\R\to\oo$ by setting
\[
\frac{\partial F}{\partial \bar{z}}:=\frac{1}{2}\left(\frac{\partial F_1}{\partial \alpha}-\frac{\partial F_2}{\partial \beta},\frac{\partial F_1}{\partial \beta}+\frac{\partial F_2}{\partial \alpha}\right)
\quad \text{ and } \quad
F'_2(\alpha,\beta):=\frac{F_2(\alpha,\beta)}{\beta}.
\]
It is immediate to verify that $\frac{\partial F}{\partial \bar{z}}$ and $(F'_2,0)$ are stem functions. Moreover, if $D\cap\R\neq\emptyset$ then $F'_2$ extends to a continuous function $G$ on the whole $D$ such that $G(\alpha,0)=\frac{\partial F_2}{\partial \beta}(\alpha,0)$ for each $(\alpha,0)\in D\cap\R$. We define the slice functions $\frac{\partial f}{\partial \bar{x}}:\Omega_D\to\oo$ and $f'_s:\Omega_D\setminus\R\to\oo$ by $\frac{\partial f}{\partial \bar{x}}:=\I\left(\frac{\partial F}{\partial \bar{z}}\right)$ and $f'_s:=\I((F'_2,0))$. {\it For short we also use the symbol $\overline{\partial}f$ to denote $\frac{\partial f}{\partial \bar{x}}$}. 
The slice function $f'_s$ is called {\it spherical derivative} of $f$. One can also define the {\it spherical value} of $f$ as the slice function $f_s^\circ:\Omega_D\to\oo$ induced by the stem function $(F_1,0)$. Note that $f(x)=f_s^\circ(x)+\mr{Im}(x)f_s'(x)$ for each $x\in\Omega_D\setminus\R$. For further information on $f_s^\circ$ and $f_s'$, we refer the reader to \cite{GPS2017,GPS2019}.

The slice function $f=\I(F):\Omega_D\to\oo$ is said to be {\it slice regular} if $\frac{\partial F}{\partial \bar{z}}=0$ on $D$ or, equivalently, if $\overline{\partial}f=0$ on $\Omega_D$. It is easy to see that the slice function $f=\I(F):\Omega_D\to\oo$ is slice regular if and only if, for some (and hence for all) $I\in\sph$, the restriction of $f$ to $\Omega_D\cap\C_I$ is holomorphic in the following sense: $\big(\frac{\partial}{\partial \alpha}+I\frac{\partial}{\partial \beta}\big)f(\alpha+I\beta)=0$. Polynomial functions $\sum_{h=0}^dx^ha_h:\oo\to\oo$ with coefficients $a_h$ in $\oo$ are examples of slice regular functions. For further details on the general stem function approach to the theory of slice regularity and on the particular case of octonions, we refer the reader to \cite{GP2011-camshaft,GP2011-new-approach,GP2011,GP2014,GP2014-global,GPR2017,GPS2017,GPS2017-singular,GPS2019,GP2019}. 


\subsection{Slice Fueter-regular functions}\label{subsec:fueter} As we said, in \cite{JRS2019} M. Jin, G. Ren and I. Sabadini introduced and studied slice Dirac-regular functions over $\oo$, we call slice Fueter-regular functions. 
Their idea 
is to mime the construction described above by decomposing the octonion algebra $\oo$ in `quaternionic slices' and by using the operator $\DD$ instead of $\frac{\partial}{\partial\overline{z}}$. Professor G. Ren presented this idea at the 12th ISAAC Congress held in Aveiro on summer 2019.

We will use notations slightly different from the ones of \cite{JRS2019}.

Let $\mc{N}:=\{(I,J)\in\sph\times\sph\,:\,I\perp J\}$ and, for each $\II=(I,J)\in\mc{N}$, let $\hh_{\II}$ be the subalgebra of~$\oo$ generated by $\{I,J\}$. Note that $\hh_{\II}$ coincides with the real vector subspace of $\oo$ generated by the orthonormal set $\{1,I,J,IJ\}$, and it is isomorphic to $\hh$. The octonion algebra $\oo$ decomposes into `quaternionic slices':
\[
\oo=\bigcup_{\II\in\mc{N}}\hh_{\II}.
\]

Identify $O(3)$ with the group of $4\times4$ orthogonal real matrices $A$ fixing the first vector of the canonical vector basis of $\R^4$, i.e. $A1=1$. Note that each matrix $A$ in $O(3)$ defines a linear trasformation both in $\R^4$ and in $\oo^4$, via the usual matrix-column multiplication $v\mapsto Av$. Here a vector $v$ in $\R^4$, or in $\oo^4$, is considered as a column; anyway, we often write $v$ as a row.

Let $E$ be a subset of $\R^4$. Suppose that it is $O(3)$-invariant, i.e. $Av\in E$ for each $v\in E$ and $A\in O(3)$. 
Define the circularization $\Theta_E$ of $E$ in $\oo$ by setting
\[
\Theta_E:=\{x_0+x_1I+x_2J+x_3IJ\in\oo\,:\,(x_0,x_1,x_2,x_3)\in E,(I,J)\in\mc{N}\}.
\]
If $E$ is open in $\R^4$, then $\Theta_E$ is open in $\oo$. Suppose that $E\neq\emptyset$. A function $\mc{F}=(\mc{F}_0,\mc{F}_1,\mc{F}_2,\mc{F}_3):E\to\oo^4$ is said to be a {\it $O(3)$-stem function} if $\mc{F}$ is $O(3)$-intrinsic in the sense that
\begin{equation} \label{eq:O3-stem}
\mc{F}(Av)=A\mc{F}(v) \quad \text{ for each $v\in E$ and $A\in O(3)$.}
\end{equation}
We denote $\mr{Stem}_{O(3)}(E,\oo^4)$ the set of all $O(3)$-stem functions from $E$ to $\oo^4$. We say that a function $f:\Theta_E\to\oo$ is a {\it (left) $O(3)$-slice function} if there exists a $O(3)$-stem function $\mc{F}=(\mc{F}_0,\mc{F}_1,\mc{F}_2,\mc{F}_3):E\to\oo^4$ such that, if $x=x_0+x_1I+x_2J+x_3IJ$ for some $v=(x_0,x_1,x_2,x_3)\in E$ and $(I,J)\in\mc{N}$, then
\[
f(x)=\mc{F}_0(v)+I\mc{F}_1(v)+J\mc{F}_2(v)+(IJ)\mc{F}_3(v).
\]
It is proven in \cite[Proposition 3.8]{JRS2019} that, thanks to \eqref{eq:O3-stem}, the preceding definition is well-posed. We say that $f$ is induced by $\mc{F}$ and we write $f=\I_{O(3)}(\mc{F})$. The $O(3)$-slice function $f$ is induced by a unique $O(3)$-stem function, see the proof of Theorem 3.10 of \cite{JRS2019}. We denote $\mc{S}_{O(3)}(\Theta_E,\oo)$ the set of all $O(3)$-slice functions from $\Theta_E$ to~$\oo$. The map $\I_{O(3)}:\mr{Stem}_{O(3)}(E,\oo^4)\to\mc{S}_{O(3)}(\Theta_E,\oo)$ is a bijection.

Suppose that $E$ is open in $\R^4$. Let $\mc{F}=(\mc{F}_0,\mc{F}_1,\mc{F}_2,\mc{F}_3):E\to\oo^4$ be a $O(3)$-stem function of class $\mscr{C}^1$ and let $f=\I_{O(3)}(\mc{F}):\Theta_E\to\oo$ be the $O(3)$-slice function induced by $\mc{F}$. According to \cite[Definition 4.1]{JRS2019}, we say that $f$ is a {\it(left $O(3)$-)slice Fueter-regular function} if $\mc{F}$ satisfies the following system:
\begin{equation} \label{eq:system}
\left\{
\begin{array}{l}
\displaystyle\frac{\partial \mc{F}_0}{\partial x_0}-\frac{\partial \mc{F}_1}{\partial x_1}-\frac{\partial \mc{F}_2}{\partial x_2}-\frac{\partial \mc{F}_3}{\partial x_3}=0
\vspace{.7em}\\
\displaystyle\frac{\partial \mc{F}_0}{\partial x_1}+\frac{\partial \mc{F}_1}{\partial x_0}-\frac{\partial \mc{F}_2}{\partial x_3}+\frac{\partial \mc{F}_3}{\partial x_2}=0 \vspace{.7em}\\
\displaystyle\frac{\partial \mc{F}_0}{\partial x_2}+\frac{\partial \mc{F}_1}{\partial x_3}+\frac{\partial \mc{F}_2}{\partial x_0}-\frac{\partial \mc{F}_3}{\partial x_1}=0
\vspace{.7em}\\
\displaystyle\frac{\partial \mc{F}_0}{\partial x_3}-\frac{\partial \mc{F}_1}{\partial x_2}+\frac{\partial \mc{F}_2}{\partial x_1}+\frac{\partial \mc{F}_3}{\partial x_0}=0\,.
\end{array}
\right.
\end{equation}

For each $\II=(I,J)\in\mc{N}$, we define the set $\Omega_{\II}:=\Omega_D\cap\hh_{\II}$, and the functions $f_{\II}:E\to\oo$ and $\DD_{\II}f:\Omega_{\II}\to\oo$ by setting
\[
f_{\II}(v):=f(x_0+x_1I+x_2J+x_3IJ)
\]
and
\[
\DD_{\II}f(x):=\frac{\partial f_{\II}}{\partial x_0}(v)+I\frac{\partial f_{\II}}{\partial x_1}(v)+J\frac{\partial f_{\II}}{\partial x_2}(v)+(IJ)\frac{\partial f_{\II}}{\partial x_3}(v)
\]
for each $v=(x_0,x_1,x_2,x_3)\in E$, where $x=x_0+x_1I+x_2J+x_3IJ$. Note that $\DD_{\II}$ is a `$\hh_{\II}$-slice' version of the Cauchy-Riemann-Fueter operator. In \cite[Proposition~4.2]{JRS2019}, it is proven that, given any $O(3)$-slice function $f=\I_{O(3)}(\mc{F})$ with $\mc{F}$ of class $\mscr{C}^1$, it holds: 
\begin{equation}\label{eq:sfrfDforall}
\text{$f$ is slice Fueter-regular if and only if $\DD_\II f=0$ on $\Omega_\II$ \textit{for all} $\II\in\mc{N}$.}
\end{equation}
Actually, one can prove some more: 
\begin{equation}\label{eq:sfrfDexists}
\text{$f$ is slice Fueter-regular if and only if $\DD_\II f=0$ on $\Omega_\II$ \textit{for at least one} $\II\in\mc{N}$.}
\end{equation}

In Section \ref{sec:proofs} below, we will give the proof of \eqref{eq:sfrfDexists}, including also a new proof of \eqref{eq:sfrfDforall}. We have decided to add such a new proof, because the original proof of \eqref{eq:sfrfDforall}, see \cite[Proposition~4.2]{JRS2019}, contains an implicit computation of $\DD_\II f$, which is not easy to make explicit in the present (non-associative) octonionic setting. See equality \eqref{eq:array''} in Remark \ref{rem:ok} below. 


\subsection{Notations} Identify $\C=\R^2$ with the real vector subspace $\R^2\times\{(0,0)\}$ of $\R^4$.

{\it Throughout the remaining part of the paper, $E$ denotes a non-empty $O(3)$-invariant open subset of $\R^4$, and $D$ the non-empty open subset of $\R^2$ defined by $D:=E\cap\R^2$.} Note that $D$ is a subset of $\R^2$ invariant under the complex conjugation. In what follows, we identify the real coordinates $(\alpha,\beta)$ and $(x_0,x_1,0,0)$ of points in $D$, used in above Subsections \ref{subsec:GS}  and \ref{subsec:fueter}, respectively; namely, we set $(\alpha,\beta)=(x_0,x_1)$.


\section{The results}\label{sec:results}

\subsection{Comparison of the theories}\label{subsec:comparison} 

Our first observation is that
\begin{equation} \label{eq:same-domain}
\Theta_E=\Omega_D.
\end{equation}
This equality is very easy to prove. Indeed, given any $x=x_0+x_1I+x_2J+x_3IJ\in \Theta_E$ with $v=(x_0,x_1,x_2,x_3)\in E$ and $(I,J)\in\mc{N}$, if $x_1=x_2=x_3=0$ then $x=x_0\in E\cap\R=D\cap\R\subset\Omega_D$. Otherwise,  $r:=(x_1^2+x_2^2+x_3^2)^{1/2}>0$ and there exists $A\in O(3)$ such that $Av=(x_0,r,0,0)$; hence $(x_0,r,0,0)\in E\cap\R^2=D$. Since $H:=\frac{1}{r}(x_1I+x_2J+x_3IJ)$ belongs to $\sph$ and $x=x_0+rH$, it follows that $x\in\Omega_D$. This shows that $\Theta_E\subset\Omega_D$. Given any $x=x_0+x_1I\in\Omega_D$ with $(x_0,x_1)\in D$ and $I\in\sph$, the point $(x_0,x_1,0,0)$ belongs to $E$ so, choosing any $J\in\sph$ with $(I,J)\in\mc{N}$, we have $x=x_0+x_1I+0J+0IJ\in\Theta_E$. This proves the inverse inclusion $\Omega_D\subset\Theta_E$ and hence the desired equality $\Theta_E=\Omega_D$.

{\it From now on, we use $\Omega_D$ to denote $\Theta_E$ also.}

If $\Omega_D$ is connected and $\Omega_D\cap\R\neq\emptyset$, then we say that $\Omega_D$ is a {\it slice domain}. If $\Omega_D$
is connected and $\Omega_D\cap\R=\emptyset$, then $\Omega_D$ is called {\it product domain}. Note that, in the first case, $D$ is connected and $D\cap\R\neq\emptyset$; in the second, $D$ has two connected components $D^+$ and $D^-$ switched by complex conjugation, and $\Omega_D$ is homeomorphic to the topological product $\sph\times D^+$. In general $\Omega_D$ decomposes into the disjoint union of its connected components, which are slice domains or product domains. In this way, from now on, we assume that
\begin{equation}\label{assumption}
\text{\it $\Omega_D$ is either a slice domain or a product domain.}
\end{equation}

Our first main result establishes the equivalence between $O(3)$-sliceness and sliceness.

\begin{thm} \label{thm:1}
A function $f:\Omega_D\to\oo$ is $O(3)$-slice if and only if it is slice, namely $\mc{S}_{O(3)}(\Omega_D,\oo)=\mc{S}(\Omega_D,\oo)$. More precisely, the following holds. Let $\Phi:\mr{Stem}_{O(3)}(E,\oo^4)\to\mr{Stem}(D,\oo^2)$ be the map sending each $O(3)$-stem function $\mc{F}=(\mc{F}_0,\mc{F}_1,\mc{F}_2,\mc{F}_3):E\to\oo^4$ into the stem function $(F_1,F_2):D\to\oo^2$ defined by 
\begin{equation}\label{eq:FFFF}
F_1(x_0,x_1):=\mc{F}_0(x_0,x_1,0,0) \quad \text{ and } \quad F_2(x_0,x_1):=\mc{F}_1(x_0,x_1,0,0).
\end{equation}

Then it holds:

\begin{itemize}
 \item[$(1)$] $\Phi$ is bijective and $\I_{O(3)}=\I \circ \Phi$. Equivalently, the following is a commutative diagram of bijective maps:
\begin{equation}\label{eq:diagram}
\xymatrix{
\mr{Stem}_{O(3)}(E,\oo^4)\ar[r]^{\Phi}\ar[d]_{\I_{O(3)}}&\mr{Stem}(D,\oo^2)\ar[d]^{\I}\\
\mc{S}_{O(3)}(\Theta_E,\oo)\ar@{=}[r]&\mc{S}(\Omega_D,\oo)
}
\end{equation}

 \item[$(2)$] If $F=(F_1,F_2):D\to\oo^2$ is a stem function and $\mc{F}=(\mc{F}_0,\mc{F}_1,\mc{F}_2,\mc{F}_3)=\Phi^{-1}(F)$, then given any $v=(x_0,x_1,x_2,x_3)\in E\setminus\R$ it holds
\begin{equation}\label{eq:FF}
\left\{
\begin{array}{l}
\mc{F}_0(v)=F_1(x_0,r),\vspace{.7em}\\
\displaystyle\mc{F}_h(v)=\frac{x_h}{r}F_2(x_0,r)\quad \text{ for each $h\in\{1,2,3\}$},
\end{array}
\right.
\end{equation}
where $r:=(x_1^2+x_2^2+x_3^2)^{1/2}$. In particular, we have:
\begin{equation}\label{eq:C-omega}
\text{$F$ is real analytic if and only if $\mc{F}=\Phi^{-1}(F)$ is.}
\end{equation}
\end{itemize}
\end{thm}

A relevant consequence is as follows.

\begin{thm}\label{thm:real-analyticity}
All slice Fueter-regular functions are real analytic. Equivalently, the stem function inducing any slice Fueter-regular function is real analytic.
\end{thm}

The concept of slice Fueter-regularity can be explicitly described in `purely slice' terms; this is the aim of our third main result.

\begin{thm} \label{thm:2}
Let $f:\Omega_D\to\oo$ be a slice function and let $F=(F_1,F_2):D\to\oo^2$ be the stem function inducing $f$. The following assertions are equivalent.
\begin{enumerate}
 \item[$(1)$] $f$ is slice Fueter-regular, i.e. $\mc{F}=\Phi^{-1}(F)$ is of class $\mscr{C}^1$ and satisfies \eqref{eq:system}.
 \item[$(2)$] $F_1$ is of class $\mscr{C}^2$, $F_2$ is of class $\mscr{C}^3$ and they satisfy the following Vekua system on $D\setminus\R$:
\begin{equation} \label{eq:variant-CR}
\left\{
\begin{array}{l}
\displaystyle
\frac{\partial F_1}{\partial x_0}(x_0,x_1)-\frac{\partial F_2}{\partial x_1}(x_0,x_1)=\frac{2F_2(x_0,x_1)}{x_1}
\vspace{.7em}\\
\displaystyle
\frac{\partial F_1}{\partial x_1}(x_0,x_1)+\frac{\partial F_2}{\partial x_0}(x_0,x_1)=0\,.
\end{array}
\right.
\end{equation}
 \item[$(2\mr{a})$] $F_1$ is of class $\mscr{C}^2$, $F_2$ is of class $\mscr{C}^3$ and $\overline\partial f=f'_s$ on $\Omega_D\setminus \R$.
 \item[$(2\mr{b})$] $F_1$ is of class $\mscr{C}^2$ and there exists a function $G_2:D\to\oo$ of class $\mscr{C}^3$ such that $F_2(x_0,x_1)=x_1G_2(x_0,x_1)$ and
\begin{equation} \label{eq:variant-CR1}
\left\{
\begin{array}{l}
\displaystyle
\frac{\partial F_1}{\partial x_0}(x_0,x_1)-x_1\frac{\partial G_2}{\partial x_1}(x_0,x_1)=3G_2(x_0,x_1)
\vspace{.7em}\\
\displaystyle
\frac{\partial F_1}{\partial x_1}(x_0,x_1)+x_1\frac{\partial G_2}{\partial x_0}(x_0,x_1)=0.
\end{array}
\right.
\end{equation}
on $D$.
 \item[$(2\mr{c})$] If $D^*$ denotes the set $\{(x_0,x_1)\in\R^2\,:\,x_1\geq0,(x_0,\sqrt{x_1})\in D\}$, then there exist an open neighborhood $D^{**}$ of $D^*$ in $\R^2$ and two functions $H_1,H_2:D^{**}\to\oo$ such that $H_1$ is of class $\mscr{C}^2$, $H_2$ is of class $\mscr{C}^3$, and $F_1(x_0,x_1)=H_1(x_0,x_1^2)$, $F_2(x_0,x_1)=x_1H_2(x_0,x_1^2)$ and
\begin{equation} \label{eq:variant-CR2}
\left\{
\begin{array}{l}
\displaystyle
\frac{\partial H_1}{\partial x_0}(x_0,x_1)-2x_1\frac{\partial H_2}{\partial x_1}(x_0,x_1)=3H_2(x_0,x_1)
\vspace{.7em}\\
\displaystyle
2\frac{\partial H_1}{\partial x_1}(x_0,x_1)+\frac{\partial H_2}{\partial x_0}(x_0,x_1)=0.
\end{array}
\right.
\end{equation}
on the open subset $D^*$ of $\R^2$.
\end{enumerate}

Furthermore, if the preceding assertions are verified, then $F$, $\mc{F}$ and $f$ are real analytic.   
\end{thm}

The reader notes that, thanks to \eqref{eq:FF}, above system \eqref{eq:variant-CR1} coincides exactly with (4.5) of \cite{JRS2019}. Furthermore, if $\Omega_D$ is a product domain, then in the preceding statement one can replace `of class $\mscr{C}^2$' and `of class $\mscr{C}^3$' with `of class $\mscr{C}^1$'.

In point $(2)$ of the preceding statement, we referred to \eqref{eq:variant-CR} as a Vekua system. For the theory of Vekua systems, see \cite[Section 6.2]{GHS2016} and references mentioned therein. Actually, Vekua systems of form \eqref{eq:variant-CR} characterize axially monogenic functions of degree zero. For the theory of axially monogenic functions of degree zero and also of positive degree, we refer the reader to \cite{LB1983,So1984,CS1993,So2000} and \cite[Chapter II, Section 2.4]{DSS1992}. For recent developments on this topic, see \cite{CSS2013,PPSS2017,DKQS2019} and their references.

Two consequences of Theorem \ref{thm:2} are as follows:

\begin{cor} \label{cor:3}
Only the constant functions are both slice Fueter-regular and slice regular.
\end{cor}

\begin{cor}\label{cor:4}
The slice product $f\cdot g$ of two slice Fueter-regular functions $f,g:\Omega_D\to\oo$ is again a slice Fueter-regular function if and only if either $f$ is constant or $g$ is.
%
\end{cor}

\begin{remark}
Let $F=(F_1,F_2):D\to\R^2\subset\oo^2$ be a stem function such that $F_1$ is of class $\mscr{C}^2$, $F_2$ is of class $\mscr{C}^3$ and they satisfy differential system~\eqref{eq:variant-CR}. Theorem \ref{thm:2} ensures that $F$ is real analytic. As a matter of fact, we have that the theory of slice Fueter-regularity allows to prove standard regularity results for $\mscr{C}^3$-solutions of a $2\times 2$ differential system as \eqref{eq:variant-CR}.~\bs
\end{remark}

\begin{quest}\label{quest:1}
Is the argument based on the theory of slice Fueter-regular functions, described in above Remark, an example of a general strategy to prove regularity results for solutions of certain real differential systems, via the embedding of $\R$ into higher dimensional real algebras, as $\oo$?
\end{quest}

\begin{remark}
As we have just said, the stem function $F=(F_1,F_2):D\to\oo^2$ inducing a slice Fueter-regular function is real analytic. In particular, both $F_1$ and $F_2$ are of class~$\mscr{C}^2$. In this way, we can differentiate both the equations of \eqref{eq:variant-CR} with respect to $\frac{\partial}{\partial x_0}$ and $\frac{\partial}{\partial x_1}$. We easily deduce:
\[\textstyle
\Delta F_1=-\frac{2}{x_1}\,\frac{\partial F_1}{\partial x_1}
\quad\text{ and }\quad
\Delta F_2=\frac{2}{x_1^2}F_2-\frac{2}{x_1}\,\frac{\partial F_2}{\partial x_1} \quad \text{ on $D\setminus\R$. \bs}
\]
\end{remark}

We recall that, if $f=\I(F)$ is slice regular, then $F$ is harmonic, i.e. $\Delta F_1=\Delta F_2=0$. Making use of the harmonicity of $F_2$, A. Perotti proved in \cite[Section 6]{Pe2019} that all quaternionic slice regular functions are biharmonic.

\begin{quest}\label{quest:2}
Have slice Fueter-regular functions some harmonicity properties, similar to the biharmonicity of quaternionic slice regular functions?
\end{quest}

Let us present some examples of slice Fueter-regular functions.

\begin{examples}
(1) Let $(F_1,F_2):\R^2\setminus\R\to\oo^2$ be a stem function of class $\mscr{C}^1$ and let $f:\oo\setminus\R\to\oo$ be the slice function induced by $(F_1,F_2)$. Define $\R_+:=\{t\in\R\,:\,t>0\}$. Suppose that there exist four functions $\gamma,\delta,\eta,\xi:\R_+\to\oo$ of class $\mscr{C}^1$ such that
\begin{align}
F_1(x_0,x_1)&=\gamma(x_0)+\delta(x_1),\label{1}\\
F_2(x_0,x_1)&=x_1\eta(x_0)+x_1\xi(x_1)\label{2}
\end{align} 
if $x_1>0$. 
Thanks to equations \eqref{eq:variant-CR} and some elementary computations, one easily deduces that $f=\I((F_1,F_2))$, with $(F_1,F_2)$ of the forms \eqref{1} and \eqref{2}, is slice Fueter-regular if and only if there exist octonionic constants $a,b,c,d$ such that
\begin{align}
F_1(x_0,x_1)&=(3x_0^2-x_1^2)a+3x_0b+c,\label{eq:3x0}\\
F_2(x_0,x_1)&\textstyle=x_1\left(2x_0a+b+\frac{d}{x_1^3}\right)\label{eq:x12x0}
\end{align}
if $x_1>0$. The slice Fueter-regular function $f=\I((F_1,F_2)):\oo\setminus\R\to\oo$ induced by the stem function defined by \eqref{eq:3x0} and \eqref{eq:x12x0} has the following form:
\begin{equation}\label{ex1}
\textstyle
f(x)=(3\mr{Re}(x)^2+2\mr{Re}(x)\mr{Im}(x)+\mr{Im}(x)^2)a+3\mr{Re}(x)b+\mr{Im}(x)b+c+\frac{\mr{Im}(x)}{|\mr{Im}(x)|^3}\,d.
\end{equation}
Note that, if $d=0$, then the above function $f$ extends to a slice Fueter-regular function on the whole $\oo$. Such extensions include the example given in \cite[Example 4.4]{JRS2019}.

Equations \eqref{eq:3x0} and \eqref{eq:x12x0} imply that a slice Fueter-regular function $f:\oo\setminus\R\to\oo$ depends only on $\mr{Im}(x)$, i.e. it is of the form $f(x)=\varphi(\mr{Im}(x))$ for some $\varphi:\oo\setminus\R\to\oo$ of class $\mscr{C}^1$ if and only if $f(x)=c+\frac{\mr{Im}(x)}{|\mr{Im}(x)|^3}\,d$ for some $c,d\in\oo$.

(2) If in the preceding example we replace \eqref{2} with the following
\begin{equation}\label{3}
F_2(x_0,x_1)=\eta(x_0)+x_1\xi(x_1)\quad \text{ for $x_1>0$},
\end{equation}
then $f=\I((F_1,F_2))$, with $(F_1,F_2)$ of the forms \eqref{1} and \eqref{3}, is slice Fueter-regular if and only if there exist $b,c,d\in\oo$ such that
\begin{align}
F_1(x_0,x_1)&=3x_0b+c,\label{eq:3x0b}\\
F_2(x_0,x_1)&\textstyle=x_1b+\frac{d}{x_1^2}\label{eq:x1b}
\end{align}
if $x_1>0$. Note that equations \eqref{eq:3x0b} and \eqref{eq:x1b} are particular cases of \eqref{eq:3x0} and \eqref{eq:x12x0}. However, \eqref{eq:3x0b} and \eqref{eq:x1b} prove that a slice Fueter-regular function $f:\oo\setminus\R\to\oo$ depends only on $\mr{Re}(x)$ if and only if it is constant.

(3) Slice Fueter powers associated to a slice Fueter-regular function, we will introduce in Definitions \ref{def:sfpf} and \ref{def:sfrpf} below, are examples of homogeneous slice Fueter-regular functions, see Corollaries \ref{cor:taylor} and \ref{cor:laurent} below. \bs 
\end{examples}


For the comparison between Fueter-regular functions and slice regular functions on Clifford algebras $\R_n=C\ell(0,n)$ for $n\geq2$, we refer the reader to papers \cite{Pe2011,Pe2013,Pe2019}.


\subsection{Slice Fueter operator}\label{subsec:sfo}

Given $p,q,h\in\N$ with $p>0$ and $q>0$, and a non-empty open subset $O$ of $\R^p$, we denote $\mscr{C}^h(O,\R^q)$ the set of all functions from $O$ to $\R^q$ of class $\mscr{C}^h$.

Let us rename the real vector basis $(1,i,j,k,\ell,\ell i,\ell j,\ell k)$ of $\oo$ by $(1,e_1,e_2,e_3,e_4,e_5,e_6,e_7)$, i.e.
\[
(e_1,e_2,e_3,e_4,e_5,e_6,e_7):=(i,j,k, \ell, \ell i, \ell j,\ell k).
\]

For each $m,n\in\{1,\ldots,7\}$ with $m<n$, we denote $L_{mn}:\mscr{C}^1(\oo,\oo)\to\mscr{C}^0(\oo,\oo)$ the usual spherical tangential operator defined by
\[
L_{mn}:=x_m\frac{\partial}{\partial x_n}-x_n\frac{\partial}{\partial x_m}.
\]
With regard to the preceding definition of $L_{mn}$, we implicitly mean that, given any non-empty open subset $O$ of $\oo$ and any $f\in\mscr{C}^1(O,\oo)$, $L_{mn}(f)$ denotes the function in $\mscr{C}^0(O,\oo)$ given by $L_{mn}(f)(x)=x_m\frac{\partial f}{\partial x_n}(x)-x_n\frac{\partial f}{\partial x_m}(x)$ for each $x\in O$. Note that, if $O'$ is a non-empty open subset of $O$, then $L_{mn}(f)=0$ on $O'$ if and only if $L_{mn}(f|_{O'})=0$, i.e. $L_{mn}(f|_{O'})$ is costantly equal to zero on $O'$. The same clarification is valid for all differential operators we will consider below.  

Given any $f\in\mscr{C}^1(\Omega_D,\oo)$ and any $x\in \Omega_D\setminus\R$, the differential at $x$ of the restriction $f|_{\sph_x}$ of $f$ to $\sph_x$ vanishes if and only if $\nabla f(x)\in\mr{Span}(1,x)$, which is in turn equivalent to say that $L_{mn}(f)(x)=0$ for each $n,m\in\{1,\ldots,7\}$ with $m<n$. Since $\sph_x$ is connected, it follows that $f|_{\sph_x}$ is constant if and only if 
\begin{equation}\label{eq:Lmn}
\text{$L_{mn}(f)=0$ on $\sph_x$ for each $n,m\in\{1,\ldots,7\}$ with $m<n$.}    
\end{equation}

The spherical Dirac operator is a well-known differential operator used in Clifford analysis, see \cite[Section 8.7]{BDS1982} and \cite[Chapter II, Section 0.1.4, p.139]{DSS1992}. See also \cite[Sections 3,4,6]{Pe2019} for a recent development.

In the next definition, we introduce the concept of spherical Dirac operator over octonions: the defining formula is essentially the usual one, but now we have to choose the `right' position of the parentheses. To the best of our knowledge, the concept we are going to introduce is new.

\begin{defn}\label{def:sdo}
We define the \emph{(octonionic) spherical Dirac operator} $\Gamma:\mscr{C}^1(\oo,\oo)\to\mscr{C}^0(\oo,\oo)$ by setting
\[
\Gamma:=-\sum_{m<n}e_m(e_nL_{mn}),
\]
where the symbol `$\sum_{m<n}$' denotes `$\sum_{m,n\in\{1,\ldots,7\},m<n}$'. \bs
\end{defn}

Note that, if $f\in\mscr{C}^1(\Omega_D,\oo)$, then
\[
\Gamma(f)(x)=-\sum_{m<n}e_m\left(e_n\left(x_m\frac{\partial f}{\partial x_n}(x)-x_n\frac{\partial f}{\partial x_m}(x)\right)\right)
\]
for each $x=(x_0,\ldots,x_7)\in\Omega_D$.

Making use of the first Moufang identity (see \cite[(3.4), p. 28]{Sch1966}), we obtain:

\begin{lem}\label{lem:moufang}
Let $f:\Omega_D\to\oo$ be a slice function of class $\mscr{C}^1$. Then it holds:
\begin{align}\label{eq:6}
f_s'={\textstyle\frac{1}{6}}\mr{Im}^{-1}\Gamma(f) \quad  \text{ on $\Omega_D\setminus\R$},
\end{align}
where `$\,\mr{Im}^{-1}\Gamma(f)$' denotes the function $\Omega_D\setminus\R\to\oo$, $x\mapsto(\mr{Im}(x))^{-1}(\Gamma(f)(x))$.
\end{lem}

It is worth noting that, given $f$ as in the statement of the preceding lemma, $f_s^\circ=f-\frac{1}{6}\Gamma(f)$ on $\Omega_D$. In the case $f(x)=\mr{Im}(x)$, \eqref{eq:6} reduces to $\Gamma(\mr{Im}(x))=6\mr{Im}(x)$ on $\oo$. 

As a consequence of Lemma \ref{lem:moufang}, we obtain the following {\it differential} sliceness criterion; for other sliceness criterions, see \cite[Lemma 3.2]{GP2014-global} and Lemma \ref{lem:sliceness-criterion} below.

\begin{lem}[Differential sliceness criterion]\label{lem:dsc}
Let $f:\Omega_D\to\oo$ be a function of class $\mscr{C}^2$. The following assertions are equivalent.
\begin{itemize}
 \item[(1)] $f$ is a slice function.
 \item[(2)] $f|_{\Omega_D\setminus\R}\in\bigcap_{m<n}\ker\big(L_{mn}(\mr{Im}^{-1}\Gamma)\big)\cap\bigcap_{m<n}\ker\big(L_{mn}(\mr{Id}-\frac{1}{6}\Gamma)\big)$ or, equivalently
\begin{equation}
\left\{
\begin{array}{ll}\label{eq:dsc}
L_{mn}(\mr{Im}^{-1}\Gamma(f))=0 & \text{for each $m,n\in\{1,\ldots,7\}$ with $m<n$,}
\vspace{.7em}\\
L_{mn}(f-{\textstyle\frac{1}{6}}\Gamma(f))=0 & \text{for each $m,n\in\{1,\ldots,7\}$ with $m<n$}
\end{array}
\right.
\end{equation}
on $\Omega_D\setminus\R$.
\end{itemize}
\end{lem}

Let us introduce the concept of slice Fueter operator. 

\begin{defn}
We define the \emph{slice Fueter operator $\overline{\vartheta}_F:\mscr{C}^1(\oo\setminus\R,\oo)\to\mscr{C}^0(\oo\setminus\R,\oo)$} as follows:
\begin{equation}\label{eq:thetaf}
\overline{\vartheta}_F:=\frac{\partial}{\partial x_0}-\mr{Im}^{-1}\mr{E}-{\textstyle\frac{1}{6}}\mr{Im}^{-1}\Gamma,
\end{equation}
where $\mr{E}$ is the Euler operator $\sum_{h=1}^7x_h\frac{\partial}{\partial x_h}$. \bs
\end{defn}

Note that $\frac{\partial}{\partial x_0}-\mr{Im}^{-1}\mr{E}$ is the operator $\overline\vartheta$ introduced in \cite{GP2014-global}, so we have:
\begin{equation}\label{eq:thetaf'}
\overline{\vartheta}_F=\overline{\vartheta}-{\textstyle\frac{1}{6}}\mr{Im}^{-1}\Gamma.
\end{equation}

Given any $f\in\mscr{C}^1(\Omega_D,\oo)$, we have:
\[
\overline{\vartheta}_F(f)=\frac{\partial f}{\partial x_0}(x)-\mr{Im}(x)^{-1}\left(\sum_{h=1}^7x_h\frac{\partial f}{\partial x_h}(x)\right)-\mr{Im}(x)^{-1}(\Gamma(f)(x)).
\]

In Theorem~2.2(ii) of \cite{GP2014-global}, it is proven that, if $f$ is slice, then $\overline{\partial}f=\overline\vartheta(f)$ on $\Omega_D\setminus\R$. Combining the latter equality with equivalence $(1)\Longleftrightarrow(2\mr{a})$ of Theorem \ref{thm:2} and with Lemma \ref{lem:dsc}, we obtain:

\begin{thm}\label{thm:sfo}
Let $f:\Omega_D\to\oo$ be a function of class $\mscr{C}^3$. The following assertions are equivalent:
\begin{itemize}
 \item[$(1)$] $f$ is a slice Fueter-regular function.
 \item[$(2)$] $f$ is a slice function and $\overline{\vartheta}_F(f)=0$ on $\Omega_D\setminus\R$.
 \item[$(3)$] $f$ satisfies \eqref{eq:dsc} and $\overline{\vartheta}_F(f)=0$ on $\Omega_D\setminus\R$.
\end{itemize}
\end{thm}

The reader observes that above equivalence $(1)\Longleftrightarrow(3)$ allows to describe the slice Fueter-regular functions $f$ as the $\mscr{C}^2$-solutions of the second order differential system given by \eqref{eq:dsc} and $\overline{\vartheta}_F(f)=0$, without any request concerning the sliceness of $f$.


\subsection{Slice Borel-Pompeiu and Cauchy integral formulas}\label{subsec:cauchy} {\it Fix $I\in\sph$.} 
Let $\pi_I:\oo\to\C_I$ be the orthogonal projection of $\oo$ onto $\C_I$. For each $\xi\in\oo$, we define $\xi_I:=\pi_I(\xi)$, $\xi_I^\perp:=\xi-\xi_I$, the slice function $g_{I,\xi}:\oo\to\oo$ by
\begin{equation}\label{eq:gIxi}
g_{I,\xi}(x):=(\xi_I-x)\cdot(\overline{\xi_I}-\overline{x})+|\xi_I^\perp|^2=|x|^2-x\overline{\xi_I}-\overline{x}\xi_I+|\xi|^2\,,
\end{equation}
the subset $\sph_{I,\xi}$ of $\oo$ by 
\[
\sph_{I,\xi}:=
\left\{
\begin{array}{ll}
\displaystyle
\sph_\xi & \text{if $\xi\in\C_I$}\vspace{.5em}\\
\displaystyle
\emptyset & \text{if $\xi\not\in\C_I$}
\end{array}
\right.
\]
and the closed subset $\Gamma_I$ of $\oo^2$ by
\[
\Gamma_I:=\{(x,\xi)\in\oo^2\,:\,\xi\in\C_I,x\in\sph_\xi\}=\bigcup_{\xi\in\oo}(\sph_{I,\xi}\times\{\xi\}).
\]
Note that $(\oo^2\setminus\Gamma_I)\cap(\oo\times\{\xi\})=(\oo\setminus\sph_{I,\xi})\times\{\xi\}$ for each $\xi\in\oo$.

\begin{lem} \label{lem:NgIxi}
Given any $\xi\in\oo$, the following assertions hold.
\begin{enumerate}
 \item $N(g_{I,\xi})(x)=|\Delta_{\xi_I}(x)|^2+2|\xi_I^\perp|^2\big(|x|^2-2\mr{Re}(x)\mr{Re}(\xi_I)+|\xi_I|^2\big)+|\xi_I^\perp|^4\geq0$ for each $x\in\oo$.
 \item The zero set of $N(g_{I,\xi})$ is equal to $\sph_{I,\xi}$.
\end{enumerate}
\end{lem}

{\it Fix $J\in\sph$ with $I\perp J$ and define $\II:=(I,J)\in\mc{N}$.} Recall that $\Omega_{\II}=\Omega_D\cap\hh_{\II}$. Assume that $D$ is bounded and its boundary $\partial D$ in $\R^2$ is of class $\mscr{C}^1$. It follows that $\Omega_{\II}$ is bounded and its boundary $\partial\Omega_{\II}$ in $\hh_{\II}$ is of class~$\mscr{C}^1$ as well. We denote $\n_{\II}:\partial\Omega_{\II}\to\hh_{\II}\subset\oo$ the outer unit normal vector field of $\partial\Omega_{\II}$ in $\hh_{\II}$. We indicate $\mr{ds}_{\II}(\xi)$ the surface element of $\partial\Omega_{\II}$ and $\mr{dv}_{\II}(\xi)$ the volume element of $\Omega_{\II}\subset\hh_{\II}$.

{\it Fix a slice function $f:\Omega_D\to\oo$ of class $\mscr{C}^1$.} Thanks to Lemma \ref{lem:NgIxi}, given any $\xi\in\partial\Omega_{\II}$, we can define the slice function $S_{f,\II,\xi}:\oo\setminus\sph_{I,\xi}\to\oo$ by setting
\[
S_{f,\II,\xi}(x):=\frac{1}{2\pi^2}(g_{I,\xi}(x)^{-\sss\bullet})^{{\sss\bullet\,}2}\cdot\big((\overline{\xi}-\overline{x})(\n_{\II}(\xi)f(\xi))\big)
\]
or, equivalently,
\begin{equation}\label{eq:SfIxi}
S_{f,\II,\xi}(x):=\frac{1}{2\pi^2}(N(g_{I,\xi})(x))^{-2}\cdot(g^c_{I,\xi}(x))^{{\sss\bullet\,}2}\cdot\big((\overline{\xi}-\overline{x})(\n_{\II}(\xi)f(\xi))\big),
\end{equation}
where $(g_{I,\xi}(x)^{-\sss\bullet})^{{\sss\bullet\,}2}$ and $(g^c_{I,\xi}(x))^{{\sss\bullet\,}2}$ denote the slice squares of $g_{I,\xi}^{-\sss\bullet}$ and $g_{I,\xi}^c$ respectively, i.e. $g_{I,\xi}^{-\sss\bullet}\cdot g_{I,\xi}^{-\sss\bullet}$ and $g_{I,\xi}^c\cdot g_{I,\xi}^c$. Note that  $(N(g_{I,\xi}))^{-2}$ is a slice preserving function, so it belongs to the center of $\mc{S}(\oo\setminus\sph_{I,\xi},\oo)$. Hence it is not necessary to parenthesize the three-terms slice product in definition \eqref{eq:SfIxi}. Furthermore, for each octonionic constant $a$, the pointwise product function $x\mapsto \overline{\xi}a-\overline{x}a=(\overline{\xi}-\overline{x})a$ coincides with the slice product $x\mapsto(\overline{\xi}-\overline{x})\cdot a$. As a consequence, the function $x\mapsto(\overline{\xi}-\overline{x})(\n_{\II}(\xi)f(\xi))$ is slice and hence also the function $x\mapsto S_{f,\II,\xi}(x)$ is.

By Lemma \ref{lem:NgIxi}, given any $\xi\in\Omega_\II$, we can also define the slice function $V_{f,\II,\xi}:\oo\setminus\sph_{I,\xi}\to\oo$ by setting
\[
V_{f,\II,\xi}(x):=\frac{1}{2\pi^2}(g_{I,\xi}(x)^{-\sss\bullet})^{{\sss\bullet\,}2}\cdot\big((\overline{\xi}-\overline{x})(\DD_{\II}f(\xi))\big)
\]
or, equivalently,
\begin{equation}\label{eq:VfIxi}
V_{f,\II,\xi}(x):=\frac{1}{2\pi^2}(N(g_{I,\xi})(x))^{-2}\cdot(g^c_{I,\xi}(x))^{{\sss\bullet\,}2}\cdot\big((\overline{\xi}-\overline{x})(\DD_{\II}f(\xi))\big).
\end{equation}

We define
\begin{equation}\label{eq:partial*}
\partial^*\Omega_{\II}:=(\oo^2\setminus\Gamma_I)\cap(\oo\times\partial\Omega_{\II})=\bigcup_{\xi\in\partial\Omega_{\II}}((\oo\setminus\sph_{I,\xi})\times\{\xi\})
\end{equation}
and
\begin{equation}\label{eq:omega*}
\Omega_{\II}^*:=(\oo^2\setminus\Gamma_I)\cap(\oo\times\Omega_{\II})=\bigcup_{\xi\in\Omega_{\II}}((\oo\setminus\sph_{I,\xi})\times\{\xi\}).
\end{equation}

\begin{defn} \label{def:integrands}
We define the \emph{slice Cauchy surface $(f,\II)$-integrand} $S_{f,\II}:\partial^*\Omega_{\II}\to\oo$ and the \emph{slice Cauchy volume $(f,\II)$-integrand} $V_{f,\II}:\Omega_\II^*\to\oo$  by
\begin{equation}\label{eq:SfI}
S_{f,\II}(x,\xi):=S_{f,\II,\xi}(x)
\end{equation}
and
\begin{equation}\label{eq:VfI}
V_{f,\II}(x,\xi):=V_{f,\II,\xi}(x). \text{ \bs}
\end{equation}
\end{defn}

Next result deals with pointwise-product formulas for slice Cauchy integrands $S_{f,\II}$ and~$V_{f,\II}$. 

\begin{lem}\label{lem:MN}
Let $Q_I:\oo^3\to\oo$, $M_I:\oo^3\to\oo$ and $N_I:\oo^2\to\oo$ be the functions defined by
\begin{align*}
Q_I(x,\xi,a):=\,&(|x|^2+|\xi|^2)^2a-2(|x|^2+|\xi|^2)(x(\xi_Ia)+\overline{x}(\overline{\xi_I}a))+\\
&+x^2(\xi_I^2a)+\overline{x}^2(\overline{\xi_I}^2a)+2|x|^2|\xi_I|^2a,\\
M_I(x,\xi,a):=\,&Q_I(x,\xi,\overline{\xi}a)-\overline{x}Q_I(x,\xi,a),\\
N_I(x,\xi):=\,&|\Delta_{\xi_I}(x)|^2+2|\xi_I^\perp|^2\big(|x|^2-2\mr{Re}(x)\mr{Re}(\xi_I)+|\xi_I|^2\big)+|\xi_I^\perp|^4.
\end{align*}

Then the zero set of $N_I$ coincides with $\Gamma_I$, and it holds:
\begin{equation}\label{eq:SfII}
S_{f,\II}(x,\xi)=\frac{M_I(x,\xi,\n_{\II}(\xi)f(\xi))}{2\pi^2N_I(x,\xi)^2} 
\quad \text{ for each $(x,\xi)\in\partial^*\Omega_\II$}
\end{equation}
and
\begin{equation}\label{eq:VfII}
V_{f,\II}(x,\xi)=\frac{M_I(x,\xi,\DD_{\II}f(\xi))}{2\pi^2N_I(x,\xi)^2}
\quad \text{ for each $(x,\xi)\in\Omega_\II^*$.}
\end{equation}
\end{lem}

Note that, for each fixed $\xi$, the functions $x\mapsto S_{f,\II}(x,\xi)$ and $x\mapsto V_{f,\II}(x,\xi)$ turn out to be slice, and rational with respect to the eight real components of $x=(x_0,\ldots,x_7)\in\oo$.

\begin{remark}
Note that, due to the non-associativity of $\oo$, the 
slice Cauchy integrands $S_{f,\II}$ and $V_{f,\II}$ cannot be written as the pointwise product of a slice function in $x$ (the `potential' octonionic slice Cauchy kernel) and the functions $\xi\mapsto\n_{\II}(\xi)f(\xi)$ and $\xi\mapsto\DD_{\II}f(\xi)$, respectively. In this pointwise-product sense, we can say that an octonionic slice Cauchy kernel, independent from $f$, does not exist. This is a typical issue of the non-associative case, see \cite{GP2011,GPR2017}. \bs
\end{remark}

Given any subset $S$ of $\oo$, we denote $\mr{cl}(S)$ the Euclidean closure of $S$ in $\oo=\R^8$. 

Our slice Borel-Pompeiu and Cauchy integral formulas read as follows.

\begin{thm} \label{thm:Cauchy-int-f}
Let $f:\Omega_D\to\oo$ be a slice function of class $\mscr{C}^1$, and let $\II\in\mc{N}$. Suppose that $D$ is bounded with boundary of class $\mscr{C}^1$, and $f$ admits a continuous extension on $\mr{cl}(\Omega_D)$. Then the following integral formula for $f$ holds:
\begin{equation}\label{eq:BP}
f(x)=\int_{\partial\Omega_{\II}}S_{f,\II}(x,\xi)\,\mr{ds}_{\II}(\xi)-\int_{\Omega_{\II}}V_{f,\II}(x,\xi)\,\mr{dv}_{\II}(\xi)
\quad \text{ for each $x\in\Omega_D$.}
\end{equation}

In particular, if $f$ is slice Fueter-regular then
\begin{equation}\label{eq:C}
f(x)=\int_{\partial\Omega_{\II}}S_{f,\II}(x,\xi)\,\mr{ds}_{\II}(\xi)
\quad \text{ for each $x\in\Omega_D$}
\end{equation}
and
\begin{equation}\label{eq:C0}
0=\int_{\partial\Omega_{\II}}S_{f,\II}(x,\xi)\,\mr{ds}_{\II}(\xi)
\quad \text{ for each $x\in\oo\setminus\mr{cl}(\Omega_D)$.}
\end{equation}
\end{thm}

\begin{quest}\label{quest:3}
Let $f$ and $D$ be as in the statement of the preceding theorem, and let $I\in\sph$. Define: $D_I:=\Omega_D\cap\C_I$, $\partial D_I$ as the boundary of $D_I$ in $\C_I$, $\partial^* D_I:=(\oo^2\setminus\Gamma_I)\cap(\oo\times\partial D_I)$ and $D_I^*:=(\oo^2\setminus\Gamma_I)\cap(\oo\times D_I)$. Do there exist two real analytic functions $S_{f,I}:\partial^* D_I\to\oo$ and $V_{f,I}:D_I^*\to\oo$ such that: each $\xi$-restriction $x\mapsto S_{f,I}(x,\xi)$ is a slice Fueter-regular function, each $\xi$-restriction $x\mapsto V_{f,I}(x,\xi)$ is a slice function, $V_{f,I}=0$ if $f$ is slice Fueter-regular, and the following `curve-plane' Borel-Pompeiu formula holds:
\[
f(x)=\int_{\partial D_I}S_{f,I}(x,\xi)\,\mr{d\xi}-\int_{D_I}V_{f,I}(x,\xi)\,({\textstyle\frac{1}{2}}I\mr{d\xi\wedge d\overline{\xi}})
\quad \text{ for each $x\in\Omega_D$?}
\]
\end{quest}


\subsection{Slice Taylor expansion}\label{subsec:taylor-expansions} 
Consider $\C_i=\{x_0+ix_1\in\hh\,:\,x_0,x_1\in\R\}\subset\hh$ and write $\xx=x_0+ix_1$. Denote $\zz_1,\zz_2,\zz_3:\C_i\to\hh$ the quaternionic Fueter variables restricted to $\C_i$, i.e.
\begin{align*}
&\zz_1(\xx):=-\xx i=x_1-ix_0,\\
&\zz_2(\xx):=-\mr{Re}(\xx)j=-x_0j,\\
&\zz_3(\xx):=-\mr{Re}(\xx)k=-x_0k.
\end{align*}

Note that
\begin{equation}\label{eq:commut}
\zz_2\zz_1=\overline{\zz_1}\zz_2, \qquad \zz_3\zz_1=\overline{\zz_1}\zz_3, \qquad \zz_3\zz_2=-\zz_2\zz_3.
\end{equation}

Let $\kappa=(\kappa_1,\kappa_2,\kappa_3)\in\N^3$, let $|\kappa|:=\kappa_1+\kappa_2+\kappa_3$ and let $\kappa!:=\kappa_1!\kappa_2!\kappa_3!$. Suppose that $\kappa\neq(0,0,0)$ and hence $|\kappa|\geq1$. Consider the sequence of indices $k_1, \ldots,k_{|\kappa|}$ such that the first $\kappa_1$ indices are equal to $1$, the next $\kappa_2$ indices are equal to $2$ and the last $\kappa_3$ indices are equal to~$3$. Let $\Sigma(\kappa)$ be the group of permutations of $\{1,2,\ldots,|\kappa|\}$. For each $\sigma\in\Sigma(\kappa)$, we denote $\zz_{\kappa,\sigma}:\C_i\to\hh$ the polynomial function given by setting
\[
\zz_{\kappa,\sigma}:=\zz_{k_{\sigma(1)}}\zz_{k_{\sigma(2)}}\cdots\zz_{k_{\sigma(|\kappa|)}}.
\]
Let $\mr{Q}_8$ be the quaternion group $\{1,i,j,k,-1,-i,-j,-k\}$. Thanks to \eqref{eq:commut}, there exist, and are unique, $a(\kappa,\sigma),b(\kappa,\sigma)\in\N$ and $L_{\kappa,\sigma}\in\mr{Q}_8$ such that $\kappa_1=a(\kappa,\sigma)+b(\kappa,\sigma)$ and
\[
\zz_{\kappa,\sigma}(\xx)=\mr{Re}(\xx)^{\kappa_2+\kappa_3}\xx^{a(\kappa,\sigma)}\overline{\xx}^{b(\kappa,\sigma)}L_{\kappa,\sigma}=x_0^{\kappa_2+\kappa_3}(x_0+ix_1)^{a(\kappa,\sigma)}(x_0-ix_1)^{b(\kappa,\sigma)}L_{\kappa,\sigma}.
\]

We define the polynomial function $\Pp_\kappa:\C_i\to\hh$ by $\Pp_\kappa:=\frac{1}{\kappa!}\sum_{\sigma\in\Sigma(\kappa)}\zz_{\kappa,\sigma}$ or, equivalently,
\begin{equation}\label{def:Pp}
\Pp_\kappa(x_0+ix_1):=x_0^{\kappa_2+\kappa_3}\sum_{\sigma\in\Sigma(\kappa)}(x_0+ix_1)^{a(\kappa,\sigma)}(x_0-ix_1)^{b(\kappa,\sigma)}\frac{L_{\kappa,\sigma}}{\kappa!}.
\end{equation}

For convention, we set $\Pp_{(0,0,0)}:=1$. Note that, if $P_\kappa:\hh\to\hh$ is the standard $\kappa^{\mr{th}}$ Fueter polynomial in $\hh$ (see \cite[Definition 6.1]{GHS2008}), then $\Pp_\kappa$ is the restriction of $P_\kappa$ on $\C_i$, i.e. we have:
\begin{equation}\label{eq:fueter-polynomials-restr}
\Pp_\kappa=P_\kappa|_{\C_i}.
\end{equation}

The next result follows immediately from definition \eqref{def:Pp}.

\begin{lem}\label{lem:18}
Let $\kappa\in\N^3$ and let $\kk:=|\kappa|$. Then there exists, and is unique, a sequence $\{\ell_{\kappa,h}\}_{h=0}^\kk$ of $\kk+1$ octonions such that $\Pp_\kappa(x_0+ix_1)=\sum_{h=0}^\kk x_0^{\kk-h}(ix_1)^h\ell_{\kappa,h}$ or, equivalently,
\begin{equation*}
\Pp_\kappa(\x)=\sum_{h=0}^\kk \mr{Re}(\x)^{\kk-h}\mr{Im}(\x)^h\ell_{\kappa,h}.
\end{equation*}
\end{lem}

Note that, if $\lfloor r\rfloor$ denotes the integer part of $r\in\R$, then we can also write:
\begin{equation}\label{def:ellekh}
\Pp_\kappa(\x)=\sum_{h=0}^{\lfloor \frac{\kk}{2}\rfloor}\mr{Re}(\x)^{\kk-2h}\mr{Im}(\x)^{2h}\ell_{\kappa,2h}+\sum_{h=0}^{\lfloor \frac{\kk-1}{2}\rfloor}\mr{Re}(\x)^{\kk-2h-1}\mr{Im}(\x)^{2h+1}\ell_{\kappa,2h+1}.
\end{equation}

For instance, the $\Pp_\kappa$'s with $|\kappa|\leq3$ are $\Pp_{(0,0,0)}=1$ and the following:
\begin{itemize}
 \item $\Pp_{(0,0,1)}=\zz_3=-\mr{Re}(\xx)k$, $\Pp_{(0,1,0)}=\zz_2=-\mr{Re}(\xx)j$, $\Pp_{(1,0,0)}=\zz_1=-\xx i$;
 \item $\Pp_{(0,0,2)}=\zz_3^2=-\mr{Re}(\xx)^2$, $\Pp_{(0,1,1)}=0$, $\Pp_{(0,2,0)}=\zz_2^2=-\mr{Re}(\xx)^2=\Pp_{(0,0,2)}$, $\Pp_{(1,0,1)}=\mr{Re}(\zz_1)\zz_3=-\mr{Re}(\xx)\mr{Im}(\xx)j$, $\Pp_{(1,1,0)}=\mr{Re}(\zz_1)\zz_2=\mr{Re}(\xx)\mr{Im}(\xx)k$, $\Pp_{(2,0,0)}=\zz_1^2=-\xx^2$;
 \item $\Pp_{(0,0,3)}=\zz_3^3=\mr{Re}(\xx)^3k$, $\Pp_{(0,1,2)}=\frac{1}{3}\zz_2\zz_3^2=\frac{1}{3}\mr{Re}(\xx)^3j$, $\Pp_{(0,2,1)}=\frac{1}{3}\zz_2^2\zz_3=\frac{1}{3}\mr{Re}(\xx)^3k$, $\Pp_{(0,3,0)}=\zz_2^3=\mr{Re}(\xx)^3j$, $\Pp_{(1,0,2)}=\frac{1}{3}(2\zz_1+\overline{\zz_1})\zz_3^2=\frac{1}{3}\mr{Re}(\xx)^2(\mr{Re}(x)+3\mr{Im}(x))i$, $\Pp_{(1,1,1)}=0$, $\Pp_{(1,2,0)}=\frac{1}{3}(2\zz_1+\overline{\zz_1})\zz_2^2=\frac{1}{3}\mr{Re}(\xx)^2(\mr{Re}(x)+3\mr{Im}(x))i=\Pp_{(1,0,2)}$, $\Pp_{(2,0,1)}=\frac{1}{3}(\zz_1^2+\overline{\zz_1}^2+|\zz_1|^2)\zz_3=\frac{1}{3}\mr{Re}(\xx)(\mr{Re}(\xx)^2+3\mr{Im}(\xx)^2)k$, $\Pp_{(2,1,0)}=\frac{1}{3}(\zz_1^2+\overline{\zz_1}^2+|\zz_1|^2)\zz_2=\frac{1}{3}\mr{Re}(\xx)(\mr{Re}(\xx)^2+3\mr{Im}(\xx)^2)j$, $\Pp_{(3,0,0)}=\zz_1^3=\xx^3 i$. 
\end{itemize}
Note that the $\Pp_\kappa$'s are not $\R$-linearly independent; for instance, $\Pp_{(0,1,1)}=\Pp_{(1,1,1)}=0$, $\Pp_{(0,2,0)}=\Pp_{(0,0,2)}$ and $\Pp_{(1,2,0)}=\Pp_{(1,0,2)}$. It would be important to solve the following problem:

\begin{quest}\label{quest:real-vector-basis}
Let $\kk$ be any natural number, $\N^3_\kk:=\{\kappa\in\N^3\,:\,|\kappa|=\kk\}$ and let $\mathsf{P}_\kk$ be the real vector subspace of $\mscr{C}^0(\C_i,\oo)$ generated by the set $\{\Pp_\kappa\}_{\kappa\in\N^3_\kk}$. Find a subset $B_\kk$ of $\N^3_\kk$ and, for each $\lambda\in\N^3_\kk\setminus B_\kk$, real coefficients $\{r_{\lambda,\kappa}\}_{\kappa\in B_\kk}$ such that: $\{\Pp_\kappa\}_{\kappa\in B_\kk}$ is a real vector basis of $\mathsf{P}_\kk$ and, for each $\lambda\in\N^3_\kk\setminus B_\kk$, $\Pp_\lambda=\sum_{\kappa\in B_\kk}r_{\lambda,\kappa}\Pp_\kappa$.
\end{quest}

Thanks to the preceding list, the coefficients $\ell_{\kappa,h}$ for $|\kappa|\leq2$ are as follows:
\begin{itemize}
 \item $\ell_{(0,0,0),0}=1$;
 \item $\ell_{(0,0,1),0}=-k$, $\ell_{(0,0,1),1}=0$; $\ell_{(0,1,0),0}=-j$, $\ell_{(0,1,0),1}=0$; $\ell_{(1,0,0),0}=-i$, $\ell_{(1,0,0),1}=-i$;
 \item $\ell_{\kappa,0}=-1$, $\ell_{\kappa,1}=\ell_{\kappa,2}=0$ if $\kappa\in\{(0,0,2),(0,2,0)\}$; $\ell_{(0,1,1),0}=\ell_{(0,1,1),1}=\ell_{(0,1,1),2}=0$; $\ell_{(1,0,1),0}=0$, $\ell_{(1,0,1),1}=-j$, $\ell_{(1,0,1),2}=0$; $\ell_{(1,1,0),0}=0$, $\ell_{(1,1,0),1}=k$, $\ell_{(1,1,0),2}=0$; $\ell_{(2,0,0),0}=-1$, $\ell_{(2,0,0),1}=-2$, $\ell_{(2,0,0),2}=-1$.
\end{itemize}

{\it Fix $\II=(I,J)\in\mc{N}$.} We indicate $\psi_\II:\hh\to\oo$ the real algebra embedding such that $\psi_{\II}(1)=1$, $\psi_{\II}(i)=I$, $\psi_{\II}(j)=J$ and hence $\psi_{\II}(k)=IJ$. Given $\kappa\in\N^3$ and $h\in\{0,\ldots,|\kappa|\}$, we define $\ell_{\II,\kappa,h}\in\hh_{\II}$ by setting
\begin{equation}\label{eq:lIIk}
\ell_{\II,\kappa,h}:=\psi_{\II}(\ell_{\kappa,h}).
\end{equation}
For each $a,b\in\oo$, we denote $\langle I;a,b\rangle$ the following octonion:
\begin{equation}\label{eq:I-product}
\langle I;a,b\rangle:=-I((Ia)b).
\end{equation}
Furthermore, if $\kk:=|\kappa|$, we define the function $\Pp_{\II,\kappa}:\oo^2\to\oo$ as follows:
\begin{align}\label{eq:sfpf}
\Pp_{\II,\kappa}(x,c):=\sum_{h=0}^{\lfloor \frac{\kk}{2}\rfloor}\mr{Re}(x)^{\kk-2h}\mr{Im}(x)^{2h}(\ell_{\II,\kappa,2h}c)+\sum_{h=0}^{\lfloor \frac{\kk-1}{2}\rfloor}\mr{Re}(x)^{\kk-2h-1}\mr{Im}(x)^{2h+1}\langle I;\ell_{\II,\kappa,2h+1},c\rangle.
\end{align}

Note that, for each $c\in\oo$, the $c$-restriction $x\mapsto\Pp_{\II,\kappa}(x,c)$ is a variant of the slice product between $x\mapsto\Pp_\kappa(x)$ and the constant $c$.  

Let $f:\Omega_D\to\oo$ be a slice function of class $\mscr{C}^{|\kappa|}$. Recall that $f_{\II}:E\to\oo$ is the function defined by $f_{\II}(v):=f(x_0+x_1I+x_2J+x_3IJ)$ if $v=(x_0,x_1,x_2,x_3)\in E$. Since $f_{\II}$ is of class~$\mscr{C}^{|\kappa|}$, we can define $\partialIIk f:\Omega_\II\to\oo$ by setting
\[
\partialIIk f(x):=\frac{\partial^{|\kappa|}f_{\II}}{\partial x_1^{\kappa_1}\partial x_2^{\kappa_2}\partial x_3^{\kappa_3}}(\psi_{\II}^{-1}(x)).
\]
 
Let us introduce the slice Fueter counterpart of Taylor powers at the origin, associated to~$f$.

\begin{defn}\label{def:sfpf}
Let $\kk\in\N$. For each sequence $\mc{C}=\{c_\kappa\}_{\kappa\in\N^3}$ in $\oo$, we define the \emph{slice Fueter $(\II,\kk)$-power associated to $\mc{C}$} as the slice function $\Pp_{\II,\kk;\mc{C}}:\oo\to\oo$ given by
\[
\Pp_{\II,\kk;\mc{C}}(x):=\sum_{\kappa\in\N^3,|\kappa|=\kk}\Pp_{\II,\kappa}(x,c_\kappa).
\]
In the case $c_\kappa=\frac{\partialIIk f(0)}{\kappa!}$ for each $\kappa\in\N^3$, we call $\Pp_{\II,\kk;\mc{C}}$ the \emph{slice Fueter $(\II,\kk)$-power associated to~$f$ and centered at $0$}, and we write $\Pp_{\II,\kk;f,0}$ in place of $\Pp_{\II,\kk;\mc{C}}$. In other words, we set:
\begin{align*}
\Pp_{\II,\kk;f,0}(x):=&\,\sum_{\kappa\in\N^3,|\kappa|=\kk}\frac{1}{\kappa!}\Pp_{\II,\kappa}(x,\partialIIk f(0))=\sum_{\kappa\in\N^3,|\kappa|=\kk}\Pp_{\II,\kappa}\big(x,{\textstyle\frac{\partialIIk f(0)}{\kappa!}}\big)=\\
=&\,\sum_{\kappa\in\N^3,|\kappa|=\kk}
\left(
\sum_{h=0}^{\lfloor \frac{\kk}{2}\rfloor}\mr{Re}(x)^{\kk-2h}\mr{Im}(x)^{2h}\big(\ell_{\II,\kappa,2h}\big({\textstyle\frac{\partialIIk f(0)}{\kappa!}}\big)\big)+\right.\\
&\,+\left.\sum_{h=0}^{\lfloor \frac{\kk-1}{2}\rfloor}\mr{Re}(x)^{\kk-2h-1}\mr{Im}(x)^{2h+1}\big\langle I;\ell_{\II,\kappa,2h+1},{\textstyle\frac{\partialIIk f(0)}{\kappa!}}\big\rangle
\right).\, \text{ \bs}
\end{align*}
\end{defn}


{\it Fix $r>0$}. We denote $B(r)$ the open ball of $\oo$ centered at $0$ with radius $r$.

\begin{thm}\label{thm:taylor}
Each slice Fueter-regular function $f:B(r)\to\oo$ can be expanded as follows:
\[
f(x)=\sum_{\kk\in\N}\Pp_{\II,\kk;f,0}(x)\quad \text{ for each $x\in B(r)$,}
\]
where the series converges uniformly on each compact subset of $B(r)$.
\end{thm}

\begin{cor}\label{cor:taylor}
Let $f:B(r)\to\oo$ be a slice Fueter-regular function. Then, for each $\kk\in\N$, the slice Fueter $(\II,\kk)$-power $\Pp_{\II,\kk;f,0}:\oo\to\oo$ is slice Fueter-regular.
\end{cor}




\begin{example}
Let $a,b\in\oo$ and let $f:\oo\to\oo$ be the slice Fueter-regular function we considered in \eqref{ex1} with $c=d=0$, i.e.
\[
f(x)=3\mr{Re}(x)b+\mr{Im}(x)b+3\mr{Re}(x)^2a+2\mr{Re}(x)\mr{Im}(x)a+\mr{Im}(x)^2a.
\]
Choose $\II=(i,j)$ in $\mc{N}$. By homogeneity's reasons, it is evident that $\Pp_{(\II,0;f,0)}=0$, $\Pp_{\II,1;f,0}(x)=3\mr{Re}(x)b+\mr{Im}(x)b$, $\Pp_{\II,2;f,0}(x)=3\mr{Re}(x)^2a+2\mr{Re}(x)\mr{Im}(x)a+\mr{Im}(x)^2a$; moreover, by Corollary \ref{cor:taylor}, these functions are slice Fueter-regular.

We would like to verify the latter assertions by direct computations. First, we note that
\[
f_{\II}(x_0+ix_1+jx_2+kx_3)=3x_0b+(3x_0^2-x_1^2-x_2^2-x_3^2)a+(ix_1+jx_2+kx_3)(2x_0a+b).
\]
Hence\vspace{.3em}
\begin{itemize}
 \item $\frac{\partial_{\,\II,(0,0,0)}f(0)}{(0,0,0)!}=f(0)=0$,\vspace{.3em}
 \item $\frac{\partial_{\,\II,(0,0,1)}f(0)}{(0,0,1)!}=kb$, $\frac{\partial_{\,\II,(0,1,0)}f(0)}{(0,1,0)!}=jb$, $\frac{\partial_{\,\II,(1,0,0)}f(0)}{(1,0,0)!}=ib$,\vspace{.3em}
 \item $\frac{\partial_{\,\II,(0,0,2)}f(0)}{(0,0,2)!}=\frac{\partial_{\,\II,(0,2,0)}f(0)}{(0,2,0)!}=\frac{\partial_{\,\II,(2,0,0)}f(0)}{(2,0,0)!}=-a$, $\frac{\partial_{\,\II,(0,1,1)}f(0)}{(0,1,1)!}=\frac{\partial_{\,\II,(1,0,1)}f(0)}{(1,0,1)!}=\frac{\partial_{\,\II,(1,1,0)}f(0)}{(1,1,0)!}=0$\vspace{.3em}
\end{itemize}
and, by Artin's theorem,\vspace{.3em}
\begin{itemize}
 \item $\Pp_{\II,(0,0,0)}\big(x,\frac{\partial_{\,\II,(0,0,0)}f(0)}{(0,0,0)!}\big)=0$,\vspace{.3em}
 \item $\Pp_{\II,(0,0,1)}\big(x,\frac{\partial_{\,\II,(0,0,1)}f(0)}{(0,0,1)!}\big)=\mr{Re}(x)(-k)(kb)=\mr{Re}(x)b$,\vspace{.3em}
\\ $\Pp_{\II,(0,1,0)}\big(x,\frac{\partial_{\,\II,(0,1,0)}f(0)}{(0,1,0)!}\big)=\mr{Re}(x)(-j)(jb)=\mr{Re}(x)b$,\vspace{.3em}
\\ $\Pp_{\II,(1,0,0)}\big(x,\frac{\partial_{\,\II,(1,0,0)}f(0)}{(1,0,0)!}\big)=\mr{Re}(x)(-i)(ib)+\mr{Im}(x)\langle i;-i,ib\rangle=\mr{Re}(x)b+\mr{Im}(x)(-i((i(-i))(ib)))=\mr{Re}(x)b+\mr{Im}(x)b$,\vspace{.5em}
 \item $\Pp_{\II,(0,0,2)}\big(x,\frac{\partial_{\,\II,(0,0,2)}f(0)}{(0,0,2)!}\big)=\Pp_{\II,(0,2,0)}\big(x,\frac{\partial_{\,\II,(0,2,0)}f(0)}{(0,2,0)!}\big)=\mr{Re}(x)^2(-1)(-a)=\mr{Re}(x)^2a$, \vspace{.3em}
\\ $\Pp_{\II,(2,0,0)}\big(x,\frac{\partial_{\,\II,(2,0,0)}f(0)}{(2,0,0)!}\big)=\mr{Re}(x)^2(-1)(-a)+\mr{Im}(x)^2(-1)(-a)+\mr{Re}(x)\mr{Im}(x)\langle i,-2;-a\rangle=\mr{Re}(x)^2a+\mr{Im}(x)^2a+2\mr{Re}(x)\mr{Im}(x)a$, \vspace{.3em}
\\ $\Pp_{\II,(0,1,1)}\big(x,\frac{\partial_{\,\II,(0,1,1)}f(0)}{(0,1,1)!}\big)=\Pp_{\II,(1,0,1)}\big(x,\frac{\partial_{\,\II,(1,0,1)}f(0)}{(1,0,1)!}\big)=\Pp_{\II,(1,1,0)}\big(x,\frac{\partial_{\,\II,(1,1,0)}f(0)}{(1,1,0)!}\big)=0$.\vspace{.3em}
\end{itemize}

As a consequence, we have: $\Pp_{\II,0;f,0}(x)=0$, $\Pp_{\II,1;f,0}(x)=\sum_{\kappa\in\N^3,|\kappa|=1}\Pp_{\II,\kappa}\big(x,{\textstyle\frac{\partialIIk f(0)}{\kappa!}}\big)=3\mr{Re}(x)b+\mr{Im}(x)b$ and $\Pp_{\II,2;f,0}(x)=\sum_{\kappa\in\N^3,|\kappa|=2}\Pp_{\II,\kappa}\big(x,{\textstyle\frac{\partialIIk f(0)}{\kappa!}}\big)=3\mr{Re}(x)^2a+2\mr{Re}(x)\mr{Im}(x)a+\mr{Im}(x)^2a$, as claimed.

Finally, if $(F_1,F_2)$ and $(G_1,G_2)$ are the stem functions inducing $\Pp_{\II,1;f,0}$ and $\Pp_{\II,2;f,0}$ respectively, then we have: $F_1(x_0,x_1)=3x_0b$, $F_2(x_0,x_1)=x_1b$, $G_1(x_0,x_1)=3x_0^2a-x_1^2a$ and $G_2(x_0,x_1)=2x_0x_1a$. Consequently, $\frac{\partial F_1}{\partial x_0}-\frac{\partial F_2}{\partial x_1}=2b=2\frac{F_2}{x_1}$, $\frac{\partial F_1}{\partial x_1}+\frac{\partial F_2}{\partial x_0}=0$, and $\frac{\partial G_1}{\partial x_0}-\frac{\partial G_2}{\partial x_1}=4x_0a=2\frac{G_2}{x_1}$, $\frac{\partial G_1}{\partial x_1}+\frac{\partial G_2}{\partial x_0}=0$. In other words, $\Pp_{\II,1;f,0}$ and $\Pp_{\II,2;f,0}$ are slice Fueter-regular, as claimed. \bs
\end{example}

\begin{remark}\label{rem:general-real-case-taylor}
Let $y\in\R$, let $B_y(r):=y+B(r)$, let $f:B_y(r)\to\oo$ be a slice Fueter-regular function and let $\mc{C}_y=\{c_\kappa\}_{\kappa\in\N^3}$ be the sequence given by $c_\kappa:=\frac{\partialIIk f(y)}{\kappa!}$. We define the \emph{slice Fueter $(\II,\kk)$-power associated to~$f$ and centered at $y$} as the slice function $\Pp_{\II,\kk;f,y}:\oo\to\oo$ defined by $\Pp_{\II,\kk;f,y}(x):=\Pp_{\II,\kk;\mc{C}_y}(x-y)$. In this situation, Theorem \ref{thm:taylor} and Corollary \ref{cor:taylor} can be restated in the same way by replacing $B(r)$ and $\Pp_{\II,\kk;f,0}$ with $B_y(r)$ and $\Pp_{\II,\kk;f,y}$, respectively. \bs
\end{remark}


\subsection{Slice Laurent expansion}\label{subsec:laurent-expansions}

Let $\Env:\hh\setminus\{0\}\to\hh$ be the rational function defined by 
\[
\Env(x):=\frac{1}{2\pi^2}\frac{\overline{x}}{|x|^4}.
\]
Given any $\kappa=(\kappa_1,\kappa_2,\kappa_3)\in\N^3$, we denote $Q_\kappa:\hh\setminus\{0\}\to\hh$ the $\partial_\kappa$-derivative $\partial_\kappa\Env$
of $\Env$, i.e.
\begin{equation}\label{eq:Qkappa}
Q_\kappa(x):=\partial_\kappa\Env(x)=
\frac{\partial^{|\kappa|}\Env}{\partial x_1^{\kappa_1}\partial x_2^{\kappa_2}\partial x_3^{\kappa_3}}(x).
\end{equation}

Next result is an immediate consequence of Proposition 7.27 of \cite{GHS2008}.

\begin{lem}\label{lem:19}
Let $\x\in\C_i\setminus\{0\}\subset\hh$, let $\kappa\in\N^3$ and let $\kk:=|\kappa|$. Then there exists, and is unique, a sequence $\{m_{\kappa,h}\}_{h=0}^{\kk+1}$ of $\kk+2$ octonions such that
\[
Q_\kappa(\x)=\frac{1}{|\x|^{4+2\kk}}\sum_{h=0}^{\kk+1}\mr{Re}(\x)^{\kk+1-h}\mr{Im}(\x)^hm_{\kappa,h}
\]
or, equivalently,
\begin{equation*}
Q_\kappa(\x)=\frac{1}{|\x|^{4+2\kk}}\sum_{h=0}^{\lfloor \frac{\kk+1}{2}\rfloor}\mr{Re}(\x)^{\kk+1-2h}\mr{Im}(\x)^{2h}m_{\kappa,2h}+\frac{1}{|\x|^{4+2\kk}}\sum_{h=0}^{\lfloor \frac{\kk}{2}\rfloor}\mr{Re}(\x)^{\kk-2h}\mr{Im}(\x)^{2h+1}m_{\kappa,2h+1}.
\end{equation*}
\end{lem}

Above {\it Question} \ref{quest:real-vector-basis} can be repeated replacing the $\Pp_\kappa$'s with the $\Qq_\kappa$'s. 

{\it Fix $\II=(I,J)\in\mc{N}$.} For $h\in\{0,\ldots,\kk+1\}$, let $m_{\II,\kappa,h}$ be the octonion in $\hh_{\II}$ given by
\begin{equation}\label{eq:mIIk}
m_{\II,\kappa,h}:=\psi_{\II}(m_{\kappa,h}).
\end{equation}

We define the function $\Qq_{\II,\kappa}:(\oo\setminus\{0\})\times\oo\to\oo$ by setting
\begin{align}\label{eq:QqIIk}
\Qq_{\II,\kappa}(x,c):=\,&\frac{1}{|x|^{4+2\kk}}\sum_{h=0}^{\lfloor \frac{\kk+1}{2}\rfloor}\mr{Re}(x)^{\kk+1-2h}\mr{Im}(x)^{2h}(m_{\II,\kappa,2h}c)+\\
&+\frac{1}{|x|^{4+2\kk}}\sum_{h=0}^{\lfloor \frac{\kk}{2}\rfloor}\mr{Re}(x)^{\kk-2h}\mr{Im}(x)^{2h+1}\langle I;m_{\II,\kappa,2h+1},c\rangle.\nonumber
\end{align}

\begin{defn}\label{def:sfrpf}
Let $\kk\in\N$. For each sequence $\mc{A}=\{a_\kappa\}_{\kappa\in\N^3}$ in $\oo$, we define the \emph{slice Fueter $(\II,\kk)$-reciprocal power associated to $\mc{A}$} as the slice function $\Qq_{\II,\kk;\mc{A}}:\oo\setminus\{0\}\to\oo$ given by
\[
\Qq_{\II,\kk;\mc{A}}(x):=\sum_{\kappa\in\N^3,|\kappa|=\kk}\Qq_{\II,\kappa}(x,a_\kappa). \, \text{ \bs}
\]
\end{defn}

{\it Fix $r_1,r_2,r\in\R\cup\{+\infty\}$ with $0\leq r_1<r<r_2\leq+\infty$}. We denote $B(r_1,r_2)$ the open spherical corona of $\oo$ centered at $0$ with radii $r_1$ and $r_2$, i.e. $B(r_1,r_2)=B(r_2)\setminus(\mr{cl}(B(r_1))\cup\{0\})$. We define $S_\II(r):=\{x\in\hh_{\II}\,:\,|x|=r\}$, $\m_{\II}:S_\II(r)\to\hh_\II\subset\oo$ as the outer unit normal vector field of $S_\II(r)$ in $\hh_\II$ and $\mr{d\sigma}_\II(\xi)$ as the surface element of $S_\II(r)$. Recall that $P_\kappa$ is the standard $\kappa^{\mr{th}}$ Fueter polynomial in $\hh$ and $\psi_\II:\hh\to\oo$ is the algebra embedding such that $\psi_\II(i)=I$ and $\psi_\II(j)=J$. Define the functions $\Qq_{\II,\kappa}^*,\Pp_{\II,\kappa}^*:\sph_\II(r)\to\hh_\II\subset\oo$ by setting $\Qq_{\II,\kappa}^*(x):=\psi_\II(\Qq_\kappa(\psi_\II^{-1}(x)))$ and $\Pp_{\II,\kappa}^*(x):=\psi_\II(\Pp_\kappa(\psi_\II^{-1}(x)))$ for each $x\in\sph_\II(r)$.

\begin{thm}\label{thm:laurent}
Let $g:B(r_1,r_2)\to\oo$ be a slice Fueter-regular function. Then there exist sequences $\mc{A}=\{a_\kappa\}_{\kappa\in\N^3}$ and $\mc{B}=\{b_\kappa\}_{\kappa\in\N^3}$ in $\oo$ such that $g$ can be expanded as follows:
\[
g(x)=\sum_{\kk\in\N}\Pp_{\II,\kk;\mc{A}}(x)+\sum_{\kk\in\N}\Qq_{\II,\kk;\mc{B}}(x)\quad \text{ for each $x\in B(r_1,r_2)$,}
\]
where both the series converge uniformly on each compact subset of $B(r_1,r_2)$, and it holds
\begin{align*}
a_{\kappa}&=(-1)^{|\kappa|}\int_{S_\II(r)}\Qq_{\II,\kappa}^*(\xi)(\m_\II(\xi)g(\xi))\,\mr{d\sigma_\II}(\xi),\\
b_{\kappa}&=(-1)^{|\kappa|}\int_{S_\II(r)}\Pp_{\II,\kappa}^*(\xi)(\m_\II(\xi)g(\xi))\,\mr{d\sigma}_\II(\xi).
\end{align*}
\end{thm}

\begin{cor}\label{cor:laurent}
Let $g:B(r_1,r_2)\to\oo$, $\mc{A}$ and $\mc{B}$ be as in the statement of the preceding theorem. Then, for each $\kk\in\N$, $\Pp_{\II,\kk;\mc{A}}$ and $\Qq_{\II,\kk;\mc{B}}$ are slice Fueter-regular functions on $B(r_1,r_2)$.    
\end{cor}



\begin{remark}\label{rem:general-real-case-laurent}
Let $y\in\R$, let $B_y(r_1,r_2):=y+B(r_1,r_2)$ and let $g:B_y(r_1,r_2)\to\oo$ be a slice Fueter-regular function. Define the sequences $\mc{A}=\{a_\kappa\}_{\kappa\in\N^3}$ and $\mc{B}=\{b_\kappa\}_{\kappa\in\N^3}$ in $\oo$ by
\begin{align*}
a_{\kappa}&\textstyle
=(-1)^{|\kappa|}\int_{S_\II(r)}\Qq_{\II,\kappa}^*(\xi)(\m_\II(\xi)g(\xi+y))\,\mr{d\sigma_\II}(\xi),\\
b_{\kappa}&\textstyle
=(-1)^{|\kappa|}\int_{S_\II(r)}\Pp_{\II,\kappa}^*(\xi)(\m_\II(\xi)g(\xi+y))\,\mr{d\sigma}_\II(\xi).
\end{align*}
Then the functions $\oo\setminus\{y\}\to\oo$, $x\mapsto\Pp_{\II,\kk;\mc{A}}(x-y)$ and $x\mapsto\Qq_{\II,\kk;\mc{B}}(x-y)$ are slice Fueter-regular, the series $\sum_{\kk\in\N}\Pp_{\II,\kk;\mc{A}}(x-y)$ and $\sum_{\kappa\in\N^3}\Qq_{\II,\kk;\mc{B}}(x-y)$ converges uniformly on compact subsets of $B_y(r_1,r_2)$, and the sum of the series $\sum_{\kk\in\N}\Pp_{\II,\kk;\mc{A}}(x-y)+\sum_{\kk\in\N}\Qq_{\II,\kk;\mc{B}}(x-y)$ is equal to $g(x)$ for each $x\in B_y(r_1,r_2)$. \bs
\end{remark}

Let $y\in\Omega_D$. In paper \cite{GPS2017}, generalizing a quaternionic pioneering idea of C. Stoppato \cite{St2012}, it is proven the existence of a family $\{\mscr{S}_{y,k}:\oo\to\oo\}_{k\in\N}$ of slice regular polynomial functions with the following property: each slice regular function $f:\Omega_D\to\oo$ can be expanded as $f(x)=\sum_{k\in\N}\mscr{S}_{y,k}(x)\cdot a_k$ on certain open circular neighborhoods $U$ of $\sph_y$ in $\oo$, where $\{a_k\}_{k\in\N}$ is a sequence in $\oo$ uniquely determined by $f$, and the series converges uniformly on compact subsets of $U$. In addition, 
there exists another family $\{\mscr{S}_{y,k}:\oo\setminus\sph_y\to\oo\}_{k<0}$ of slice regular rational functions such that, given any slice regular function $g:\Omega_D\setminus\sph_y\to\oo$, $g$ can be expanded as $g(x)=\sum_{k\in\Z}\mscr{S}_{y,k}(x)\cdot b_k$ on $U\setminus\sph_y$, where the sequence $\{b_k\}_{k\in\N}$ in $\oo$ is unique and the series converges uniformly on compact subsets of $U\setminus\sph_y$. It is worth recalling that, if $y\in\R$, then $\sph_y=\{y\}$, $\mscr{S}_{y,k}$ is the usual $k^{\mr{th}}$-power centered at $y$, i.e. $\mscr{S}_{y,k}(x)=(x-y)^k$, and the neighborhoods $U$ are standard open balls of $\oo$ centered~at~$y$.

Our Theorems \ref{thm:taylor} and \ref{thm:laurent} stated above are slice Fueter-regular versions of the just mentioned results of \cite{GPS2017} in the case $y=0$, or better $y\in\R$, see Remarks \ref{rem:general-real-case-taylor} and \ref{rem:general-real-case-laurent}.

Let $y$ be any point of $\Omega_D$, and let $f:\Omega_D\to\oo$ and $g:\Omega_D\setminus\sph_y\to\oo$ be slice Fueter-regular functions.

\begin{quest}\label{quest:4}
Is it possible to generalize Theorems \ref{thm:taylor} and \ref{thm:laurent} to the case in which $y\not\in\R$? More precisely, given any $y\in\Omega_D\setminus\R$, do there exist slice Fueter-regular functions $\Pp_{\II,\kk;f,y}:\oo\to\oo$, $\Pp_{\II,\kk;\mc{A}}:\oo\to\oo$ and $\Qq_{\II,\kk;\mc{B}}:\oo\setminus\sph_y\to\oo$ for each $\kk\in\N$ such that the $\Pp_{\II,\kk;f,y}(x)$'s and the $\Pp_{\II,\kk;\mc{A}}(x)$'s are polynomial functions in the eight real components $(x_0,\ldots,x_7)$ of $x\in\oo$ of degree~$\kk$, the $\Qq_{\II,\kk;\mc{B}}(x)$'s are rational functions in $(x_0,\ldots,x_7)$ of degree $-\kk$, and these functions satisfy the following two assertions? 
\end{quest}

\begin{assertion}
There exists an open circular neighborhood $U$ of $\sph_y$ in $\oo$ such that
\[
f(x)=\sum_{\kk\in\N}\Pp_{\II,\kk;f,y}(x)\quad \text{ for each $x\in U$,}
\]
where the series converges uniformly on each compact subset of $U$.
\end{assertion}

\begin{assertion}
There exists an open circular neighborhood $U$ of $\sph_y$ in $\oo$ such that
\[
g(x)=\sum_{\kk\in\N}\Pp_{\II,\kk;\mc{A}}(x)+\sum_{\kk\in\N}\Qq_{\II,\kk;\mc{B}}(x)\quad \text{ for each $x\in U\setminus\sph_y$,}
\]
where both the series converge uniformly on each compact subset of $U\setminus\sph_y$.
\end{assertion}


\subsection{Maximum Modulus Principle}\label{subsec:mmp}

Slice Fueter-regular functions enjoy the following version of Maximum Modulus Principle. Recall that, given $f\in\mc{S}(\Omega_D,\oo)$, $V(f_s')$ denotes the zero set of the spherical derivative $f_s'$ of $f$, i.e. $V(f_s')=\{x\in\Omega_D\setminus\R\,:\,f_s'(x)=0\}$.

\begin{thm}\label{thm:MMP}
Let $f:\Omega_D\to\oo$ be a slice Fueter-regular function. Suppose that $f$ satisfies one of the following conditions:
\begin{itemize}
 \item[(1)] $|f|$ has a maximum at a point of $\Omega_D$.
 \item[(2)] $|f|$ has a local maximum at a point of $\Omega_D\setminus V(f_s')$. 
\end{itemize}
Then $f$ is constant. 
\end{thm}

It would be very interesting to know the answer to the following problem:

\begin{quest}\label{quest:38}
Does there exist a non-constant slice Fueter-regular function $f:\Omega_D\to\oo$ such that $|f|$ has a local maximum at a point of $V(f'_s)$?
\end{quest}



\section{The proofs} \label{sec:proofs}


\subsection*{Proofs of Section \ref{subsec:comparison} and proofs of assertions (\ref{eq:sfrfDforall}) and (\ref{eq:sfrfDexists})}

\begin{proof}[Proof of Theorem \ref{thm:1}]
Let $\mc{F}=(\mc{F}_0,\mc{F}_1,\mc{F}_2,\mc{F}_3):E\to\oo^4$ be a $O(3)$-stem function, and let $(F_1,F_2):D\to\oo^2$ be the map given by $F_1(x_0,x_1):=\mc{F}_0(x_0,x_1,0,0)$ and $F_2(x_0,x_1):=\mc{F}_1(x_0,x_1,0,0)$. Consider the matrices $R_1.R_2,R_3$ in $O(3)$ such that 
\begin{align*}
R_1(x_0,x_1,x_2,x_3)&=(x_0,-x_1,x_2,x_3),\\
R_2(x_0,x_1,x_2,x_3)&=(x_0,x_1,-x_2,x_3),\\
R_3(x_0,x_1,x_2,x_3)&=(x_0,x_1,x_2,-x_3).
\end{align*}
Since $\mc{F}(R_hv)=R_h\mc{F}(v)$ for each $v\in\Omega_D$ and $h\in\{1,2,3\}$, if we set $v=(x_0,x_1,0,0)\in D$, we obtain that $(F_1,F_2)$ is a stem function and $\mc{F}_2(x_1,x_2,0,0)=\mc{F}_3(x_1,x_2,0,0)=0$.

Let $x$ be any point of $\Omega_D$. Write $x=x_0+x_1I\in\Omega_D$ for some $x_0,x_1\in\R$ and $I\in\sph$. Define $v:=(x_0,x_1,0,0)\in E$ and choose $J\in\sph$ such that $I\perp J$, so $x=x_0+x_1I+0J+0(IJ)\in\Theta_E=\Omega_D$. Since $\mc{F}_2(v)=\mc{F}_3(v)=0$, it follows that
\begin{align*}
((\I\circ\Phi)(\mc{F}))(x)&=(\I(\Phi(\mc{F})))(x)=\mc{F}_0(v)+I\mc{F}_1(v)=\\
&=\mc{F}_0(v)+I\mc{F}_1(v)+J\mc{F}_2(v)+(IJ)\mc{F}_3(v)=\\
&=\I_{O(3)}(\mc{F})(x).
\end{align*}
In particular, this proves that $\mc{S}_{O(3)}(\Omega_D,\oo)\subset\mc{S}(\Omega_D,\oo)$.

Consider now $F=(F_1,F_2)\in\mr{Stem}(D,\oo^2)$ and define the map $\mc{F}=(\mc{F}_0,\mc{F}_1,\mc{F}_2,\mc{F}_3):E\to\oo^4$ by equalities \eqref{eq:FF} if $x\in E\setminus\R$, and $\mc{F}(x_0,0,0,0):=(F_1(x_0,0),0,0,0)$ for each $x_0\in E\cap\R=D\cap\R$.

We have to show that $\mc{F}$ is a $O(3)$-stem function; note that, if this is true, then $\I_{O(3)}(\mc{F})$ makes sense and it is equal to $\I(F)$, because $F=\Phi(\mc{F})$. To show that $\mc{F}$ is a $O(3)$-stem function, it suffices to prove that $\mc{F}(Bw)=B\mc{F}(w)$ for each $B\in O(3)$ and for each $w=(w_0,w_1,0,0)\in E$ with $w_1>0$. Indeed, given any $v\in E\setminus\R$, there exist $C\in O(3)$ and $w=(w_0,w_1,0,0)\in E$ with $w_1>0$ such that $v=Cw$; hence, given any $A\in O(3)$, it holds:
\[
\mc{F}(Av)=\mc{F}(A(Cw))=\mc{F}((AC)w)=(AC)\mc{F}(w)=A(C\mc{F}(w))=A\mc{F}(Cw)=A\mc{F}(v).
\]

Consider $w=(w_0,w_1,0,0)\in E$ with $w_1\geq0$, and let $w':=(w_1,0,0)\in\R^3$, so $w=(w_0,w')$. 
Choose arbitrarily $B\in O(3)$ and denote by $B^*$ the orthogonal $3\times3$ real matrix such that
\[
B=
\left(
\begin{array}{c|c}
1 & {\boldsymbol{0}}\\
\hline
{\boldsymbol{0}} & B^*
\end{array}
\right),
\]
where $\boldsymbol{0}$ denotes both a null $3$-entries row and a null $3$-entries column. Note that $Bw=(w_0,B^*w')$ and $|B^*w'|=|w'|=w_1$. Indicate $B^*_1\in\R^3$ the first column of $B^*$. 

If $w_1=0$, then $w\in\R$, $Bw=(w_0,0,0,0)=w$, and hence $\mc{F}(Bw)=\mc{F}(w)=(F_1(w_0,0),0,0,0)=B\mc{F}(w)$, as desired.

Now suppose that $w_1>0$. For each $v=(x_0,x_1,x_2,x_3)\in\R^4$, we denote $\mr{Diag}(v)$ the $4\times 4$ diagonal matrix whose diagonal entries are $x_0$, $x_1$, $x_2$ and $x_3$. Furthermore, if $v':=(x_1,x_2,x_3)$ and hence $v=(x_0,v')\in E$, we define $\mr{Y}(v)\in\oo^4$ by setting
\[
\mr{Y}(v)=\mr{Y}(x_0,v'):=(F_1(x_0,|v'|),F_2(x_0,|v'|),F_2(x_0,|v'|),F_2(x_0,|v'|)).
\]

Thanks to \eqref{eq:FF}, we know that
\begin{align*}
\mc{F}(w)&=\mr{Diag}((1,w'/|w'|))\mr{Y}(w)=\mr{Diag}((1,1,0,0))\mr{Y}(w)=\\
&=(F_1(w_0,w_1),F_2(w_0,w_1),0,0),\vspace{.7em}\\
\mc{F}(Bw)&=\mc{F}((w_0,B^*w'))=\mr{Diag}((1,B^*w'/|B^*w'|))\mr{Y}(w_0,B^*w')=\\
&=\mr{Diag}((1,B^*w'/w_1))\mr{Y}(w_0,w')=\mr{Diag}((1,B^*_1))\mr{Y}(w_0,w')=\\
&=(F_1(w_0,w_1),B^*_1F_2(w_0,w_1)).
\end{align*}
It follows that
\begin{align*}
B\mc{F}(w)=(F_1(w_0,w_1),B^*_1F_2(w_0,w_1))=\mc{F}(Bw),
\end{align*}
as desired. In particular, we have: $\mc{S}(\Omega_D,\oo)\subset\mc{S}_{O(3)}(\Omega_D,\oo)$, so $\mc{S}(\Omega_D,\oo)=\mc{S}_{O(3)}(\Omega_D,\oo)$.

Finally, we have to show that $F=(F_1,F_2)$ is real analytic if and only if $\mc{F}=(\mc{F}_0,\mc{F}_1,\mc{F}_2,\mc{F}_3)$~is. If $\mc{F}$ is real analytic, then by \eqref{eq:FFFF} it is evident that $F$ is real analytic as well. Suppose that $F$ is real analytic. Let $y\in E$. If $y\not\in\R$, then $\mc{F}$ is real analytic locally at $y$ in $E$, because the function $E\to\R$, $v=(x_0,x_1,x_2,x_3)\mapsto r(v)=(x_1^2+x_2^2+x_3^2)^{1/2}$ is. Let $y\in\R$. By \cite{Wh1943}, there exist an open neighborhood $U$ of $y$ in $D$, an open neighborhood $V$ of $y$ in $\R^2$ and real analytic functions $G_1,G_2:V\to\oo$ such that, for each $(x_0,x_1)\in U$, we have: $(x_0,x_1^2)\in V$, $F_1(x_0,x_1)=G_1(x_0,x_1^2)$ and $F_2(x_0,x_1)=x_1G_2(x_0,x_1^2)$. Then the functions $\mc{F}_0(v)=G_1(x_0,r(v)^2)$ and $\mc{F}_h(v)=x_hG_2(x_0,r(v)^2)$ for $h\in\{1,2,3\}$ are real analytic locally at $y$ in $E$, because $v\mapsto r(v)^2$ is.
\end{proof}

\begin{proof}[{Proof of Theorem \ref{thm:real-analyticity}}]
Let $f:\Omega_D\to\oo$ be a slice Fueter-regular function. By Theorem \ref{thm:1}, $f$ is a slice function. Let $(F_1,F_2)$ be the stem function inducing $f$. Choose $\II=(I,J)\in\mc{N}$. Thanks to Lemma 4.6 of \cite{JRS2019} (or equation (6.6) of the same paper), we know that the restriction of $f$ to $\Omega_\II=\Omega_D\cap\hh_{\II}$ is real analytic. In particular, this is true for the restriction of $f$ to $\Omega_D\cap\C_I$. Since $F_1(x_0,x_1)=\frac{1}{2}(f(x_0+x_1I)+f(x_0-x_1I))$ and $F_2(x_0,x_1)=-\frac{1}{2}I(f(x_0+x_1I)-f(x_0-x_1I))$ for each $(x_0,x_1)\in D$, it turns out that $F_1$ and $F_2$ are real analytic as well. Now Proposition 7(3) of \cite{GP2011} implies that $f$ is real analytic on the whole $\Omega_D$.
\end{proof}

\begin{proof}[Proof of Theorem \ref{thm:2}]
$(1)\Longrightarrow(2)$. Suppose that $f=\I(F):\Omega_D\to\oo$ is a slice Fueter-regular function. Let $\mc{F}=(\mc{F}_0,\mc{F}_1,\mc{F}_2,\mc{F}_3):E\to\oo^4$ be the $O(3)$-stem function given by $\mc{F}=\Phi^{-1}(F)$. By Theorem~\ref{thm:real-analyticity}, we know that $F$ is real analytic; in particular, $F_1$ is of class $\mscr{C}^2$ and $F_2$ of class~$\mscr{C}^3$. Moreover, if $v=(x_0,x_1,x_3,x_4)\in E\setminus\R$ and $r=(x_1^2+x_2^2+x_3)^{1/2}$, then equations \eqref{eq:FF} hold. As~a consequence, differentiating such equations, we have:
\begin{equation}\label{eq:array}
\begin{array}{l}
\displaystyle
\frac{\partial\mc{F}_0}{\partial x_0}(v)=\frac{\partial F_1}{\partial x_0}(x_0,r),\vspace{.7em}\\
\displaystyle
\frac{\partial\mc{F}_0}{\partial x_j}(v)=\frac{x_j}{r}\,\frac{\partial F_1}{\partial x_1}(x_0,r) \quad \text{ for each $j\in\{1,2,3\}$},\vspace{.7em}\\
\displaystyle
\frac{\partial\mc{F}_h}{\partial x_0}(v)=\frac{x_h}{r}\,\frac{\partial F_2}{\partial x_0}(x_0,r) \quad \text{ for each $h\in\{1,2,3\}$},\vspace{.7em}\\
\displaystyle
\frac{\partial\mc{F}_h}{\partial x_j}(v)=\frac{\delta_{hj}r^2-x_hx_j}{r^3}\,F_2(x_0,r)+\frac{x_hx_j}{r^2}\,\frac{\partial F_2}{\partial x_1}(x_0,r) \quad \text{ for each $h,j\in\{1,2,3\}$.}
\end{array}
\end{equation}
Note that $\frac{\partial\mc{F}_j}{\partial x_h}=\frac{\partial\mc{F}_h}{\partial x_j}$ for each $h,j\in\{1,2,3\}$. We have:
\begin{equation}\label{eq:array'}
\begin{array}{l}
\displaystyle
\frac{\partial \mc{F}_0}{\partial x_0}(v)-\frac{\partial \mc{F}_1}{\partial x_1}(v)-\frac{\partial \mc{F}_2}{\partial x_2}(v)-\frac{\partial \mc{F}_3}{\partial x_3}(v)=\frac{\partial F_1}{\partial x_0}(x_0,r)-\frac{\partial F_2}{\partial x_1}(x_0,r)-\frac{2}{r}\,F_2(x_0,r),\vspace{.7em}\\
\displaystyle
\frac{\partial \mc{F}_0}{\partial x_1}(v)+\frac{\partial \mc{F}_1}{\partial x_0}(v)-\frac{\partial \mc{F}_2}{\partial x_3}(v)+\frac{\partial \mc{F}_3}{\partial x_2}(v)=\frac{x_1}{r}\left(\frac{\partial F_1}{\partial x_1}(x_0,r)+\frac{\partial F_2}{\partial x_0}(x_1,r)\right),\vspace{.7em}\\
\displaystyle
\frac{\partial \mc{F}_0}{\partial x_2}(v)+\frac{\partial \mc{F}_1}{\partial x_3}(v)+\frac{\partial \mc{F}_2}{\partial x_0}(v)-\frac{\partial \mc{F}_3}{\partial x_1}(v)=\frac{x_2}{r}\left(\frac{\partial F_1}{\partial x_1}(x_0,r)+\frac{\partial F_2}{\partial x_0}(x_1,r)\right),\vspace{.7em}\\
\displaystyle
\frac{\partial \mc{F}_0}{\partial x_3}(v)-\frac{\partial \mc{F}_1}{\partial x_2}(v)+\frac{\partial \mc{F}_2}{\partial x_1}(v)+\frac{\partial \mc{F}_3}{\partial x_0}(v)=\frac{x_3}{r}\left(\frac{\partial F_1}{\partial x_1}(x_0,r)+\frac{\partial F_2}{\partial x_0}(x_1,r)\right).
\end{array}
\end{equation}

It follows at once that systems \eqref{eq:system} and \eqref{eq:variant-CR} are equivalent on $E\setminus\{x_1x_2x_3=0\}$. By density, $F_1$ and $F_2$ are solutions of system \eqref{eq:variant-CR} on the whole $D\setminus\R$.

$(2)\Longrightarrow(1)$. Suppose that $F_2$ is of class $\mscr{C}^1$, $F_2$ is of class~$\mscr{C}^3$, and $F_1$ and $F_2$ are solutions of system \eqref{eq:variant-CR} on $D\setminus\R$. By \cite{Wh1943} and \eqref{eq:FF}, $\mc{F}=\Phi^{-1}(F)$ turns out to be of class $\mscr{C}^1$. Since systems \eqref{eq:system} and \eqref{eq:variant-CR} are equivalent on $E\setminus\{x_1x_2x_3=0\}$, by density, $\mc{F}$ is a solution of system \eqref{eq:system} on the whole $E$, i.e. $f$ is slice Fueter-regular.

$(2)\Longleftrightarrow(2\mr{a})$. 
Since the slice functions $\overline\partial f=\I\big(\frac{1}{2}\big(\frac{\partial F_1}{\partial x_0}-\frac{\partial F_2}{\partial x_1},\frac{\partial F_1}{\partial x_1}+\frac{\partial F_2}{\partial x_0}\big)\big)$ and $f'_s=\I\big(\big(\frac{F_2}{x_1},0\big)\big)$ are equal if and only if their inducing stem functions $\frac{1}{2}\big(\frac{\partial F_1}{\partial x_0}-\frac{\partial F_2}{\partial x_1},\frac{\partial F_1}{\partial x_1}+\frac{\partial F_2}{\partial x_0}\big)$ and $\big(\frac{F_2}{x_1},0\big)$ are, the equivalence between points $(2)$ and $(2\mr{a})$ is evident.

$(2)\Longleftrightarrow(2\mr{b})$. Suppose that $(2)$ holds. Then $(1)$ holds as well, and hence $F_1$ and $F_2$ are real analytic. In particular,  \cite{Wh1943} implies the existence of a real analytic function $G_2:D\to\oo$ such that $F_2(x_0,x_1)=x_1G_2(x_0,x_1)$ for each $(x_0,x_1)\in D$. Since $\frac{\partial F_2}{\partial x_0}=x_1\frac{\partial G_2}{\partial x_0}$ and $\frac{\partial F_2}{\partial x_1}=G_2+x_1\frac{\partial G_2}{\partial x_1}$, equations \eqref{eq:variant-CR1} are satisfied on $D\setminus\R$, and hence on the whole $D$ by density. This proves implication $(2)\Longrightarrow(2\mr{b})$. The inverse implication is an immediate consequence of the last two equalities; indeed, we have: $\frac{\partial G_2}{\partial x_0}=\frac{1}{x_1}\frac{\partial F_2}{\partial x_0}$ and $\frac{\partial G_2}{\partial x_1}=-\frac{1}{x_1^2}F_2+\frac{1}{x_1}\frac{\partial F_2}{\partial x_1}$ on $D\setminus\R$. 

$(2)\Longleftrightarrow(2\mr{c})$. This equivalence can be proven following the same argument used to prove the preceding one. If $(2)$ holds, then $F_1$ and $F_2$ are real analytic. Hence, by \cite{Wh1943}, there exist an open neighborhood $D^{**}$ of $D^*$ in $\R^2$ and real analytic functions $H_1,H_2:D^{**}\to\oo$ such that $F_1(x_0,x_1)=H_1(x_0,x_1^2)$, $F_2(x_0,x_1)=x_1H_2(x_0,x_1^2)$ for each $(x_0,x_1)\in D$. As a consequence, we have: $\frac{\partial F_1}{\partial x_0}(x_0,x_1)=\frac{\partial H_1}{\partial x_0}(x_0,x_1^2)$, $\frac{\partial F_1}{\partial x_1}(x_0,x_1)=2x_1\frac{\partial H_1}{\partial x_1}(x_0,x_1^2)$, $\frac{\partial F_2}{\partial x_0}(x_0,x_1)=x_1\frac{\partial H_2}{\partial x_0}(x_0,x_1^2)$ and $\frac{\partial F_2}{\partial x_1}(x_0,x_1)=H_2(x_0,x_1^2)+2x_1^2\frac{\partial H_2}{\partial x_1}(x_0,x_1^2)$. It follows that $H_1$ and $H_2$ satisfy \eqref{eq:variant-CR2} on $D^*\setminus\R$, and by density on the whole $D^*$. This proves $(2\mr{c})$. Finally, if $(2\mr{c})$ is satisfied, the latter four equalities imply at once that $F_1$ and $F_2$ satisfy \eqref{eq:variant-CR}.
\end{proof}

\begin{proof}[Proofs of assertions \eqref{eq:sfrfDforall} and \eqref{eq:sfrfDexists}]
Let $f:\Omega_D\to\oo$ be a $O(3)$-slice function induced by a $O(3)$-stem function $\mc{F}=(\mc{F}_0,\mc{F}_1,\mc{F}_2,\mc{F}_3)$ of class $\mscr{C}^1$ and let $(F_1,F_2):=\Phi^{-1}(\mc{F})$. Choose $\II=(I,J)\in\mc{N}$ and $v=(x_0,x_1,x_2,x_3)\in E$. Define $K:=IJ$, $x:=x_0+x_1I+x_2J+x_3K\in\Omega_\II\setminus\R$,  $r:=(x_1^2+x_2^2+x_3^2)^{1/2}>0$, $F_2:=F_2(x_0,r)$, $a:=\frac{\partial F_1}{\partial x_0}(x_0,r)$, $b:=\frac{\partial F_1}{\partial x_1}(x_0,r)$, $c:=\frac{\partial F_2}{\partial x_0}(x_0,r)$, $d:=\frac{\partial F_2}{\partial x_1}(x_0,r)$ and $e:=d-\frac{F_2(x_0,r)}{r}$.

Thanks to \eqref{eq:array}, it holds:
\begin{align*}
\DD_\II f(x)=\,&
\left(\frac{\partial}{\partial x_0}+I\frac{\partial}{\partial x_1}+J\frac{\partial}{\partial x_2}+K\frac{\partial}{\partial x_3}\right)\!\big(\mc{F}_0+I\mc{F}_1+J\mc{F}_2+K\mc{F}_3\big)(v)=\vspace{.7em}\\
=\,&
\left(\frac{\partial\mc{F}_0}{\partial x_0}+I\frac{\partial\mc{F}_1}{\partial x_0}+J\frac{\partial\mc{F}_2}{\partial x_0}+K\frac{\partial\mc{F}_3}{\partial x_0}\right)(v)+\\
&
+I\left(\left(\frac{\partial\mc{F}_0}{\partial x_1}+I\frac{\partial\mc{F}_1}{\partial x_1}+J\frac{\partial\mc{F}_2}{\partial x_1}+K\frac{\partial\mc{F}_3}{\partial x_1}\right)(v)\right)+\\
&
+J\left(\left(\frac{\partial\mc{F}_0}{\partial x_2}+I\frac{\partial\mc{F}_1}{\partial x_2}+J\frac{\partial\mc{F}_2}{\partial x_2}+K\frac{\partial\mc{F}_3}{\partial x_2}\right)(v)\right)+\\
&
+K\left(\left(\frac{\partial\mc{F}_0}{\partial x_3}+I\frac{\partial\mc{F}_1}{\partial x_3}+J\frac{\partial\mc{F}_2}{\partial x_3}+K\frac{\partial\mc{F}_3}{\partial x_3}\right)(v)\right)=\vspace{.7em}\\
=\,&
a+\frac{x_1I}{r}c+\frac{x_2J}{r}c+\frac{x_3K}{r}c+\\
&
+I\left(\frac{x_1}{r}b+I\left(\frac{r^2-x_1^2}{r^3}F_2+\frac{x_1^2}{r^2}d\right)+\frac{x_1x_2}{r^2}Je+\frac{x_1x_3}{r^2}Ke\right)+\\
&
+J\left(\frac{x_2}{r}b+\frac{x_1x_2}{r^2}Ie+J\left(\frac{r^2-x_2^2}{r^3}F_2+\frac{x_2^2}{r^2}d\right)+\frac{x_2x_3}{r^2}Ke\right)+\\
&
+K\left(\frac{x_3}{r}b+\frac{x_1x_3}{r^2}Ie+\frac{x_2x_3}{r^2}Je+K\left(\frac{r^2-x_3^2}{r^3}F_2+\frac{x_3^2}{r^2}d\right)\right)=\vspace{.7em}\\
=\,&
a+\frac{\mr{Im}(x)}{r}(c+b)-\left(\sum_{h=1}^3\frac{r^2-x_h^2}{r^3}\right)F_2-\left(\sum_{h=1}^3\frac{x_h^2}{r^2}\right)d+\\
&
+\frac{x_1x_2}{r^2}(IJ+JI)e+\frac{x_1x_3}{r^2}(IK+KI)e+\frac{x_2x_3}{r^2}(JK+KJ)e=\vspace{.7em}\\
=\,&
\left(a-d-\frac{2}{r}F_2\right)+\frac{\mr{Im}(x)}{r}(c+b).
\end{align*}
The preceding chain of equalities proves that
\begin{equation}\label{eq:ok}
\DD_\II f(x)=\left(\frac{\partial F_1}{\partial x_0}(x_0,r)-\frac{\partial F_2}{\partial x_1}(x_0,r)-\frac{2}{r}F_2(x_0,r)\right)+\frac{\mr{Im}(x)}{r}\left(\frac{\partial F_1}{\partial x_1}(x_0,r)+\frac{\partial F_2}{\partial x_0}(x_0,r)\right)
\end{equation}
for each $x\in\Omega_\II\setminus\R$, where $x_0=\mr{Re}(x)$ and $r=|\mr{Im}(x)|$.

By Theorem \ref{thm:2}, if $f$ is slice Fueter-regular, then $a-d-\frac{2}{r}F_2=0$ and $c+b=0$. This proves that $\DD_\II f=0$ on $\Omega_\II$ for each $\II\in\mc{N}$. Conversely, if $\DD_\II f=0$ on $\Omega_\II$ for at least one element $\II=(I,J)$ of $\mc{N}$, then \eqref{eq:ok} ensures that $0=\DD_\II f(x)=\big(a-d-\frac{2}{r}F_2\big)+\frac{\mr{Im}(x)}{|\mr{Im}(x)|}(c+b)$ for each $x\in\Omega_\II\setminus\R$. Since $-x\in\Omega_\II\setminus\R$ for each $x\in\Omega_\II\setminus\R$, we deduce that $a-d-\frac{2}{r}F_2=c+b=0$. Using \eqref{eq:array'}, it follows that \eqref{eq:system} holds on $E\setminus\R$, and hence on the whole $E$ by density.

The proof is complete.
\end{proof}

\begin{remark}\label{rem:ok}
Let $f:\Omega_D\to\oo$ be a $O(3)$-slice function induced by a $O(3)$-stem function $\mc{F}=(\mc{F}_0,\mc{F}_1,\mc{F}_2,\mc{F}_3)$ of class $\mscr{C}^1$. Choose $\II=(I,J)\in\mc{N}$. Combining equality \eqref{eq:ok} with equalities \eqref{eq:array'}, we deduce:
\begin{equation}\label{eq:array''}
\begin{array}{rl}
\DD_\II f(x)=&\textstyle\!\!\!
\left(\frac{\partial \mc{F}_0}{\partial x_0}(v)-\frac{\partial \mc{F}_1}{\partial x_1}(v)-\frac{\partial \mc{F}_2}{\partial x_2}(v)-\frac{\partial \mc{F}_3}{\partial x_3}(v)\right)+\vspace{.7em}\\
&+I\left(\frac{\partial \mc{F}_0}{\partial x_1}(v)+\frac{\partial \mc{F}_1}{\partial x_0}(v)-\frac{\partial \mc{F}_2}{\partial x_3}(v)+\frac{\partial \mc{F}_3}{\partial x_2}(v)\right)+\vspace{.7em}\\
\textstyle
&+J\left(\frac{\partial \mc{F}_0}{\partial x_2}(v)+\frac{\partial \mc{F}_1}{\partial x_3}(v)+\frac{\partial \mc{F}_2}{\partial x_0}(v)-\frac{\partial \mc{F}_3}{\partial x_1}(v)\right)+\vspace{.7em}\\
\textstyle
&+(IJ)\left(\frac{\partial \mc{F}_0}{\partial x_3}(v)-\frac{\partial \mc{F}_1}{\partial x_2}(v)+\frac{\partial \mc{F}_2}{\partial x_1}(v)+\frac{\partial \mc{F}_3}{\partial x_0}(v)\right)
\end{array}
\end{equation}
for each $x=x_0+x_1I+x_2J+x_3K\in\Omega_\II$, where $v:=(x_0,x_1,x_2,x_3)\in E$. Note that \eqref{eq:array''} coincides with equality (4.4) contained in the proof of Proposition 4.2 of \cite{JRS2019}. \bs
\end{remark}

\begin{proof}[{Proof of Corollary \ref{cor:3}}]
If $f$ is slice regular and slice Fueter-regular, then $0=\frac{\partial F_1}{\partial x_0}-\frac{\partial F_2}{\partial x_1}=\frac{2F_2}{x_1}$ so $F_2=0$, $\frac{\partial F_1}{\partial x_0}=0$ and $\frac{\partial F_1}{\partial x_1}=-\frac{\partial F_2}{\partial x_0}=0$ on $D$. By assumption~\eqref{assumption}, it turn out that $F_1=c$ on $D$ for some $c\in\oo$, and hence $f=c$ on $\Omega_D$. Evidently, each constant function is both slice Fueter-regular and slice regular.
\end{proof}

Recall that, given any slice function $f=\I((F_1,F_2)):\Omega_D\to\oo$, the spherical value of $f$ is the slice function $f_s^\circ:\Omega_D\to\oo$ defined by $f_s^\circ:=\I((F_1,0))$. Let $\Ii:\Omega_D\to\oo$ be the imaginary part function, i.e. $\Ii(x):=\mr{Im}(x)$. Note that $\I$ is the slice preserving function induced by the stem function $\R^2\to\oo^2$, $(x_0,x_1)\mapsto(0,x_1)$. It is immediate to verify that $f=f_s^\circ+\Ii \cdot f_s'=f_s^\circ+\Ii f_s'$.

\begin{proof}[{Proof of Corollary \ref{cor:4}}]
Let $f=\I((F_1,F_2))$ and $g=\I((G_1,G_2))$ be slice Fueter-regular functions. By Theorem \ref{thm:2}, we know that $\overline{\partial}f=f'_s$ and $\overline{\partial}g=g'_s$ on $\Omega_D\setminus\R$. By direct computations, we see that $\overline{\partial}(f\cdot g)=(\overline{\partial}f)\cdot g+f\cdot (\overline{\partial}g)$ and $(f\cdot g)'_s=f'_s\cdot g^\circ_s+f^\circ_s\cdot g'_s$. It follows that
\begin{align*}
\overline{\partial}(f\cdot g)-(f\cdot g)'_s&=f'_s\cdot (g-g_s^\circ)+(f-f_s^\circ)\cdot g'_s=f'_s\cdot (\Ii \cdot g_s')+(\Ii \cdot f_s') \cdot g'_s=2\,\Ii (f_s'g_s').
\end{align*}
Hence $\overline{\partial}(f\cdot g)-(f\cdot g)'_s=0$ if and only if $f'_sg'_s=0$ on $\Omega_D\setminus\R$. By Theorem \ref{thm:real-analyticity}, $f$ is real analytic. Since $f_s'(x)=\mr{Im}(x)^{-1}(f(x)-f(\overline{x}))$, it follows that $f_s'$ is real analytic as well. The same is true for $g_s'$. By assumption \eqref{assumption}, $\Omega_D\setminus\R$ is connected, so the condition `$f'_sg'_s=0$ on $\Omega_D\setminus\R$' is equivalent to require that either $f_s'=0$ or $g_s'=0$ on $\Omega_D\setminus\R$. This is in turn equivalent to the condition `either $F_2=0$ or $G_2=0$ on $D$'. Note that, if $F_2=0$ on $D$, equations \eqref{eq:variant-CR} imply that $F_1$ is constant, and hence $f$ is constant as well. Similarly, if $G_2=0$ on $D$, then $g$ is constant. This completes the proof.
\end{proof}


\subsection*{Proofs of Section \ref{subsec:sfo}}

\begin{proof}[Proof of Lemma \ref{lem:moufang}]
Let $m,n\in\{1,\ldots,7\}$ with $m<n$. Since $f_s^\circ$ and $f_s'$ are constant on each $\sph_x$, we have that $L_{mn}(f_s^\circ)=L_{mn}(f_s')=0$ on the whole $\Omega_D\setminus\R$. Since $f(x)=f_s^\circ(x)+\mr{Im}(x)f_s'(x)$ for each $x\in\Omega_D\setminus\R$, it follows that
\begin{align*}
L_{mn}(f)(x)&=L_{mn}(f_s^\circ)(x)+L_{mn}(\mr{Im}(x))f_s'(x)+\mr{Im}(x)L_{mn}(f_s')(x)=\\
&=L_{mn}(\mr{Im}(x))f_s'(x)=(x_me_n-x_ne_m)f_s'(x).
\end{align*}

Let $x\in\Omega_D\setminus\R$ and let $a:=f_s'(x)$. Thanks to the first Moufang identity \cite[(3.4), p. 28]{Sch1966}, we have: $e_m(e_n(e_ma))=(e_me_ne_m)a$. Since $e_me_ne_m=-e_me_me_n=e_n$ for $m<n$, we deduce that $e_m(e_n(e_ma))=e_na$ for $m<n$ and hence, by Artin's theorem,
\begin{align*}
\Gamma(f)(x)&=-\sum_{m<n}e_m(e_nL_{mn}(f)(x))=-\sum_{m<n}e_m(e_n((x_me_n-x_ne_m)a))=\\
&=-\sum_{m<n}e_m(e_n(x_me_na-x_ne_ma))=-\sum_{m<n}e_m(-x_ma-x_ne_n(e_ma))=\\
&=-\sum_{m<n}(-x_me_ma-x_ne_m(e_n(e_ma)))=\sum_{m<n}(x_me_ma+x_ne_na)=\\
&=\left(\sum_{m<n}(x_me_m+x_ne_n)\right)a=\frac{1}{2}\left(\sum_{m,n=1}^7(x_me_m+x_ne_n)-\sum_{m=1}^7(x_me_m+x_me_m)\right)=\\
&=\frac{1}{2}\left(7\sum_{m=1}^7x_me_m+7\sum_{n=1}^7x_ne_n-2\sum_{m=1}^7x_me_m\right)=6\mr{Im}(x)a,
\end{align*}
as desired.
\end{proof}

\begin{lem}\label{lem:restriction}
A function $f:\Omega_D\to\oo$ is slice if and only if the restriction $f|_{\Omega_D\setminus\R}$ is. 
\end{lem}
\begin{proof}
We can assume $\Omega_D\cap\R=D\cap\R\neq\emptyset$. If $(F_1,F_2):D\to\oo^2$ is a stem function inducing $f:\Omega_D\to\oo$, then the restriction $(F_1,F_2)|_{D\setminus\R}$ is a stem function inducing $f|_{\Omega_D\setminus\R}$. Conversely, if $f:\Omega_D\to\oo$ is a function such that $f|_{\Omega_D\setminus\R}$ is a slice function, induced by a certain stem function $(G_1,G_2):D\setminus\R\to\oo^2$, then the function $(F_1,F_2):D\to\oo^2$, defined by $(F_1,F_2)=(G_1,G_2)$ on $D\setminus\R$, $F_1=f$ on $D\cap\R$ and $F_2=0$ on $D\cap\R$, is a stem function such that $\I((F_1,F_2))=f$.
\end{proof}

\begin{proof}[Proof of Lemma \ref{lem:dsc}]
By Lemma \ref{lem:restriction}, we can assume that $\Omega_D\cap\R=D\cap\R=\emptyset$, i.e. $\Omega_D$ is a product domain. If $f$ is slice then Lemma \ref{lem:moufang} implies that $\frac{1}{6}\mr{Im}^{-1}\Gamma(f)=f_s'$, so $f-\frac{1}{6}\Gamma(f)=f_s^\circ$. It follows that the functions $\mr{Im}^{-1}\Gamma(f)$ and $f-\frac{1}{6}\Gamma(f)$ are constant on the sphere $\sph_x$ for each $x\in\Omega_D\setminus\R$. Hence \eqref{eq:dsc} is verified.

Suppose now that \eqref{eq:dsc} is verified. Let $h\in\mscr{C}^1(\Omega_D,\oo)$ and $g\in\mscr{C}^1(\Omega_D,\oo)$ be the functions defined by $h:=\frac{1}{6}\mr{Im}^{-1}\Gamma(f)$ and $g:=f-\frac{1}{6}\Gamma(f)$. Note that $f(x)=g(x)+\mr{Im}(x)h(x)$ for each $x\in\Omega_D$. Thanks to \eqref{eq:dsc}, we know that $g$ and $h$ are constant on each spheres $\sph_x$ for $x\in\Omega_D$. Define the stem function $F=(F_1,F_2):D\to\oo$ by setting $F_1(x_0,x_1):=g(x_0+x_1i)$ and $F_2(x_0,x_1):=x_1h(x_0+x_1i)$. It remains to show that $\I(F)=f$. 
Let $x=x_0+x_1I\in\Omega_D$ for some $x_0,x_1\in\R$ and $I\in\sph$. Since $g$ and $h$ are constant on $\sph_x$, we have:
\begin{align*}
\I(F)(x)&=F_1(x_0,x_1)+IF_2(x_0,x_1)=\\
&=g(x_0+x_1i)+Ix_1h(x_0+x_1i)=\\
&=g(x_0+x_1I)+x_1Ih(x_0+x_1I)=\\
&=g(x)+\mr{Im}(x)h(x)=f(x).
\end{align*}
The proof is complete.
\end{proof}

\begin{proof}[Proof of Theorem \ref{thm:sfo}]
First, suppose that $f$ is slice Fueter-regular. By implication $(1)\Longrightarrow(2\mr{a})$ of Theorem \ref{thm:2}, we know that $\overline{\partial}f-f_s'=0$ on $\Omega_D\setminus\R$. Theorem 2.2(ii) of \cite{GP2014-global} and Lemma \ref{lem:moufang} imply that $\overline{\partial}f=\overline{\vartheta}(f)$ and $f_s'=\frac{1}{6}\mr{Im}^{-1}\Gamma(f)$ on $\Omega_D\setminus\R$, respectively. Bearing in mind \eqref{eq:thetaf'}, it follows that $\overline{\vartheta}_F(f)=0$ on $\Omega_D\setminus\R$. Since $f$ is of class $\mscr{C}^2$, thanks to implication $(1)\Longrightarrow(2)$ of Lemma \ref{lem:dsc}, we also have that $f$ satisfies \eqref{eq:dsc} on $\Omega_D\setminus\R$. This proves implications $(1)\Longrightarrow(2)$ and $(2)\Longrightarrow(3)$ of the present theorem.

Finally, suppose that $(3)$ holds, i.e. $f$ satisfies \eqref{eq:dsc} and $\overline{\vartheta}_F(f)=0$ on $\Omega_D\setminus\R$. By implication $(2)\Longrightarrow(1)$ of Lemma \ref{lem:dsc}, we know that $f$ is slice. In this way, we can apply again Theorem 2.2(ii) of \cite{GP2014-global} and Lemma \ref{lem:moufang}. Combining these results with \eqref{eq:thetaf'} once again, we deduce: $\overline{\partial}f-f_s'=\overline{\vartheta}_F(f)=0$ on $\Omega_D\setminus\R$. Since $f$ is assumed to be of class $\mscr{C}^3$, implication $(2\mr{a})\Longrightarrow(1)$ of Theorem~\ref{thm:2} ensures that $(1)$ holds, i.e. $f$ is slice Fueter-regular. The proof is complete.
\end{proof}


\subsection*{Proofs of Section \ref{subsec:cauchy}}

\begin{proof}[Proof of Lemma \ref{lem:NgIxi}]
Since $\overline{\xi_I}\in\sph_{\xi_I}$, we have $\Delta_{\overline{\xi_I}}(x)=\Delta_{\xi_I}(x)$ and hence  $\Delta_{\overline{\xi_I}}(\overline{x})=\overline{\Delta_{\xi_I}(x)}$. Bearing in mind the latter equality, it holds:
\begin{align*}
N(g_{I,\xi})(x)=\,&((\xi_I-x)\cdot(\overline{\xi_I}-\overline{x})+|\xi_I^\perp|^2)\cdot((\overline{\xi_I}-x)\cdot(\xi_I-\overline{x})+|\xi_I^\perp|^2)=\\
=\,&N(\xi_I-x)N(\overline{\xi_I}-\overline{x})+|\xi_I^\perp|^2(|\xi_I|^2-\overline{x}\xi_I-x\overline{\xi_I}+|x|^2)+\\
&+|\xi_I^\perp|^2(|\xi_I|^2-\overline{x}\overline{\xi_I}-x\xi_I+|x|^2)+|\xi_I^\perp|^4=\\
=\,&\Delta_{\xi_I}(x)\Delta_{\overline{\xi_I}}(\overline{x})+2|\xi_I^\perp|^2(|x|^2-x\mr{Re}(\xi_I)-\overline{x}\mr{Re}(\xi_I)+|\xi_I|^2)+|\xi_I^\perp|^4=\\
=\,&|\Delta_{\xi_I}(x)|^2+2|\xi_I^\perp|^2\big(|x|^2-2\mr{Re}(x)\mr{Re}(\xi_I)+|\xi_I|^2\big)+|\xi_I^\perp|^4.
\end{align*}
Note that
\[
|x|^2-2\mr{Re}(x)\mr{Re}(\xi_I)+|\xi_I|^2=|\mr{Im}(x)|^2+(\mr{Re}(x)-\mr{Re}(\xi_I))^2+|\mr{Im}(\xi_I)|^2\geq0.
\]
Hence $N(g_{I,\xi})(x)\geq0$ for each $x\in\oo$, and $N(g_{I,\xi})(x)=0$ if and only if $\xi_I^\perp=0$ (i.e. $\xi\in\C_I$, so $\xi=\xi_I$) and $x\in\sph_\xi$. The proof is complete.
\end{proof}

Let us recall the representation formula for octonionic slice functions, see \cite[Proposition~6]{GP2011} for a proof.

\begin{lem}\label{lem:grf}
Let $f:\Omega_D\to\oo$ be a slice function, and let $I,J\in\sph$ such that $I\neq J$. For each $(\alpha,\beta)\in D$ and $K\in\sph$, it holds:
\[
f(\alpha+K\beta)=(K-J)((I-J)^{-1}f(\alpha+I\beta))-(K-I)((I-J)^{-1}f(\alpha+J\beta)).
\]
\end{lem}

For each $I\in\sph$, we define $\C_I^+:=\{\alpha+I\beta\in\oo\,:\,\alpha,\beta\in\R,\beta\geq0\}$.

An immediate consequence of Lemma \ref{lem:grf} is as follows:

\begin{cor}[Identity principle]\label{cor:ip}
Let $f,g:\Omega_D\to\oo$ be two slice functions and let $I,J\in\sph$ such that $I\neq J$ and $f=g$ on $\Omega_D\cap(\C_I^+\cup\C_J^+)$. Then $f=g$ on the whole $\Omega_D$.

In particular, $f=g$ on the whole $\Omega_D$ if $f$ and $g$ coincide on $\Omega_D\cap\C_I$.
\end{cor}

The slice Cauchy integrands $S_{f,\II}$ and $V_{f,\II}$ can be characterized as follows.

\begin{lem}\label{lem:char-Cauchy}
The slice Cauchy surface $(f,\II)$-integrand $S_{f,\II}$ is the unique function $\mc{S}:\partial^*\Omega_{\II}\to\oo$ satisfying the following condition: for each $\xi\in\partial\Omega_{\II}$, the restriction function $\mc{S}|_\xi:\oo\setminus\sph_{I,\xi}\to\oo$, defined by $\mc{S}|_\xi(x):=\mc{S}(x,\xi)$, is a slice function such that
\begin{equation}\label{eq:S|}
\mc{S}|_\xi(y)=\frac{1}{2\pi^2}\frac{\overline{\xi-y}}{|\xi-y|^4}(\n_{\II}(\xi)f(\xi)) \quad \text{ for each $y\in\C_I\setminus\sph_{I,\xi}$.}
\end{equation}

Similarly, the slice Cauchy volume $(f,\II)$-integrand $V_{f,\II}$ is the unique function $\mc{V}:\Omega_\II^*\to\oo$ satisfying the following condition: for each $\xi\in\Omega_\II$, the restriction function $\mc{V}|_\xi:\oo\setminus\sph_{I,\xi}\to\oo$, defined by $\mc{V}|_\xi(x):=\mc{V}(x,\xi)$, is a slice function such that
\begin{equation}\label{eq:V|}
\mc{V}|_\xi(y)=\frac{1}{2\pi^2}\frac{\overline{\xi-y}}{|\xi-y|^4}(\DD_{\II}f(\xi)) \quad \text{ for each $y\in\C_I\setminus\sph_{I,\xi}$.}
\end{equation}
\end{lem}
\begin{proof}
Let us prove the first part of the lemma. The proof of the second is similar. During this proof, we will use `$x$' as a variable. Let $\xi\in\partial\Omega_{\II}$ and let $a:=\n_{\II}(\xi)f(\xi)$. By Corollary \ref{cor:ip}, it suffices to show that
\begin{equation}\label{eq:39'}
S_{f,\II}(y,\xi)=\frac{1}{2\pi^2}\,\frac{\overline{\xi}a-\overline{y}a}{|\xi-y|^4} \quad \text{ for each $y\in\C_I\setminus\sph_{I,\xi}$.}
\end{equation}

Choose any $y\in\C_I\setminus\sph_{I,\xi}$. Since the slice functions $x\mapsto\xi_I-x$, $x\mapsto\overline{\xi_I}-\overline{x}$ and hence $g_{I,\xi}$ are $\C_I$-preserving, it follows that
\begin{align*}
g_{I,\xi}(y)&=(\xi_I-y)(\overline{\xi_I}-\overline{y})+|\xi_I^\perp|^2=|\xi_I-y|^2+|\xi_I^\perp|^2=|\xi-y|^2,\\
g^c_{I,\xi}(y)&=|\xi_I-\overline{y}|^2+|\xi_I^\perp|^2=|\xi-\overline{y}|^2=|\overline{\xi}-y|^2,\\
N(g_{I,\xi})(y)&=|\xi-y|^2|\overline{\xi}-y|^2
\end{align*}
and hence
\begin{align*}
&((N(g_{I,\xi})^{-2}\cdot(g^c_{I,\xi})^{{\sss\bullet\,}2})(y)=(|\xi-y|^2|\overline{\xi}-y|^2)^{-2}|\overline{\xi}-y|^4=|\xi-y|^{-4},\\
&((N(g_{I,\xi})^{-2}\cdot(g^c_{I,\xi})^{{\sss\bullet\,}2}\cdot (\overline{\xi}a))(y)=|\xi-y|^{-4}\overline{\xi}a,\\
&((N(g_{I,\xi})^{-2}\cdot(g^c_{I,\xi})^{{\sss\bullet\,}2}\cdot a)(y)=|\xi-y|^{-4}a.
\end{align*}
Note that
\[
N(g_{I,\xi})^{-2}\cdot(g^c_{I,\xi})^{{\sss\bullet\,}2}\cdot (\overline{x}a)=N(g_{I,\xi})^{-2}\cdot(g^c_{I,\xi})^{{\sss\bullet\,}2}\cdot (\overline{x}\cdot a)=\overline{x}((N(g_{I,\xi})^{-2}\cdot(g^c_{I,\xi})^{{\sss\bullet\,}2}\cdot a).
\]

Summarizing, it hold:
\begin{align*}
&((N(g_{I,\xi})^{-2}\cdot(g^c_{I,\xi})^{{\sss\bullet\,}2}\cdot (\overline{\xi}a))(y)=|\xi-y|^{-4}\overline{\xi}a,\\
&(N(g_{I,\xi})^{-2}\cdot(g^c_{I,\xi})^{{\sss\bullet\,}2}\cdot (\overline{x}a))(y)=\overline{y}|\xi-y|^{-4}a=|\xi-y|^{-4}\overline{y}a.
\end{align*}

The latter two equalities imply at once that
\begin{align*}
2\pi^2S_{f,\II}(y,\xi)&=\big((N(g_{I,\xi})(x))^{-2}\cdot(g^c_{I,\xi}(x))^{{\sss\bullet\,}2}\cdot(\overline{\xi}a-\overline{x}a)\big)(y)=\\
&=|\xi-y|^{-4}\overline{\xi}a-|\xi-y|^{-4}\overline{y}a=|\xi-y|^{-4}(\overline{\xi}a-\overline{y}a),
\end{align*}
which proves \eqref{eq:39'}.
\end{proof}

\begin{proof}[Proof of Lemma \ref{lem:MN}]
As above, we will only prove the first part, the proof of the second being similar. Let $\xi\in\partial\Omega_{\II}$, let $a:=\n_{\II}(\xi)f(\xi)$ and let $y\in\C_I\setminus\sph_{I,\xi}$. Since the function $x\mapsto\frac{M_I(x,\xi,a)}{2\pi^2N_I(x,\xi)^2}$ is slice, by Corollary \ref{cor:ip} and Lemma \ref{lem:char-Cauchy}, it suffices to show that $\frac{M_I(y,\xi,a)}{N_I(y,\xi)^2}=\frac{\overline{\xi}a-\overline{y}a}{|\xi-y|^4}$. 
Bearing in mind Artin's theorem and the fact that $y$ and $\xi_I$ belong to $\C_I$, we obtain:
\begin{align*}
N_I(y,\xi)=&\,|(y-\xi_I)(y-\overline{\xi_I})|^2+|\xi_I^\perp|^2(2|y|^2-2\mr{Re}(y)2\mr{Re}(\xi_I)+2|\xi_I|^2)+|\xi_I^\perp|^4=\\
=&\,|y-\xi_I|^2|y-\overline{\xi_I}|^2+|\xi_I^\perp|^2(|y|^2-y\overline{\xi_I}-\overline{y}\xi_I+|\xi_I|^2+\\
&+|y|^2-y\xi_I-\overline{y}\overline{\xi_I}+|\xi_I|^2)+|\xi_I^\perp|^4=\\
=&\,|y-\xi_I|^2|y-\overline{\xi_I}|^2+|\xi_I^\perp|^2(|y-\xi_I|^2+|y-\overline{\xi_I}|^2)+|\xi_I^\perp|^4=\\
=&\,(|y-\xi_I|^2+|\xi_I^\perp|^2)(|y-\overline{\xi_I}|^2+|\xi_I^\perp|^2)=|y-\xi|^2|y-\overline{\xi}|^2,\\
Q_I(y,\xi,a)=&\,Q_I(y,\xi,1)a,\\
Q_I(y,\xi,1)=&\,(|y|^2+|\xi|^2)^2-2(|y|^2+|\xi|^2)(y\xi_I+\overline{y}\overline{\xi_I})+y^2\xi_I^2+\overline{y}^2\overline{\xi_I}^2+2|y|^2|\xi_I|^2=\\
=&\,(|y|^2+|\xi|^2-y\xi_I-\overline{y}\overline{\xi_I})^2=(|y|^2+|\xi_I|^2-y\xi_I-\overline{y}\overline{\xi_I}+|\xi_I^\perp|^2)^2=\\
=&\,(|y-\overline{\xi_I}|^2+|\xi_I^\perp|^2)^2=|y-\overline{\xi}|^4.
\end{align*}
As a consequence, we deduce:
\begin{align*}
M_I(y,\xi,a)&=Q_I(y,\xi,\overline{\xi}a)-\overline{y}Q_I(y,\xi,a)=Q_I(y,\xi,1)(\overline{\xi}a)-\overline{y}(Q_I(y,\xi,1)a)=|y-\overline{\xi}|^4(\overline{\xi}a-\overline{y}a)
\end{align*}
and hence
\[
\frac{M_I(y,\xi,a)}{N_I(y,\xi)^2}=\frac{|y-\overline{\xi}|^4(\overline{\xi}a-\overline{y}a)}{|y-\xi|^4|y-\overline{\xi}|^4}=\frac{\overline{\xi}a-\overline{y}a}{|\xi-y|^4},
\]
as desired.
\end{proof}


Next result is a version of Lemma 3.2 of \cite{GP2014-global}.

\begin{lem}[Sliceness criterion]\label{lem:sliceness-criterion}
Let $f:\Omega_D\to\oo$ be a function. The following assertions are equivalent.
\begin{itemize}
 \item[(1)] $f$ is a slice function.
 \item[(2)] There exists $J\in\sph$ with the following property: for each $(\alpha,\beta)\in D$ and for each $I\in\sph$, if $x=\alpha+\beta I$ and $y=\alpha+\beta J$, then
\begin{equation}\label{eq:sl-cr}
f(x)=\frac{1}{2}(f(y)+f(\overline{y}))-\frac{I}{2}(J(f(y)-f(\overline{y}))).
\end{equation}
 \item[(3)] Equation \eqref{eq:sl-cr} holds for each $x=\alpha+\beta I$ and $y=\alpha+\beta J$ with $(\alpha,\beta)\in D$ and $I,J\in\sph$.
\end{itemize}
\end{lem}
\begin{proof}
If $f$ is a slice function induced by the stem function $(F_1,F_2)$ then, given any $J\in\sph$, $F_1(\alpha,\beta)=\frac{1}{2}(f(y)+f(\overline{y}))$ and $F_2(\alpha,\beta)=-\frac{1}{2}J(f(y)-f(\overline{y}))$, so \eqref{eq:sl-cr} holds for each $I\in\sph$. This proves that (1) implies (2) and (3).

Suppose that there exists $J\in\sph$ as in (2). Define the function $(F_1,F_2):D\to\oo^2$ by setting $F_1(\alpha,\beta):=\frac{1}{2}(f(y)+f(\overline{y}))$ and $F_2(\alpha,\beta):=-\frac{1}{2}J(f(y)-f(\overline{y}))$. Note that $F_1(\alpha,-\beta)=\frac{1}{2}(f(\overline{y})+f(y))=F_1(\alpha,\beta)$ and $F_2(\alpha,-\beta)=-\frac{1}{2}J(f(\overline{y})-f(y))=-F_2(\alpha,\beta)$. Hence $(F_1,F_2)$ is a stem function. Thanks to \eqref{eq:sl-cr}, we have that $\I((F_1,F_2))=f$, so $f$ is a slice function. In particular, $f$ satisfies (3). This proves that (2) implies (1) and (3).

Evidently, (3) implies (2). The proof is complete. 
\end{proof}

\begin{cor}\label{cor:int-sliceness}
Let $f:\Omega_D\to\oo$ be a slice function of class $\mscr{C}^1$, and let $\mathsf{S},\mathsf{V}:\Omega_D\to\oo$ be the functions defined by setting
\[
\mathsf{S}(x):=\int_{\partial\Omega_{\II}}S_{f,\II}(x,\xi)\,\mr{ds}_{\II}(\xi)
\quad\text{ and }\quad
\mathsf{V}(x):=\int_{\Omega_{\II}}V_{f,\II}(x,\xi)\,\mr{dv}_{\II}(\xi).
\]
Then $\mathsf{S}$ and $\mathsf{V}$ are slice functions.
\end{cor}
\begin{proof}
By definitions \eqref{eq:SfIxi}, \eqref{eq:SfI} and \eqref{eq:VfIxi}, \eqref{eq:VfI}, for each fixed $\xi$, the functions $x\mapsto S_{f,\II}(x,\xi)$ and $x\mapsto V_{f,\II}(x,\xi)$ are slice functions, so they satisfy \eqref{eq:sl-cr} for each given $x=\alpha+\beta I$ and $y=\alpha+\beta J$ with $(\alpha,\beta)\in D$ and $I,J\in\sph$. As a consequence, we have:
\begin{align*}
\mathsf{S}(x)=\,&\int_{\partial\Omega_{\II}}S_{f,\II}(x,\xi)\,\mr{ds}_{\II}(\xi)=\\
=\,&\int_{\partial\Omega_{\II}}\left(\frac{1}{2}(S_{f,\II}(y,\xi)+S_{f,\II}(\overline{y},\xi))-\frac{I}{2}(J(S_{f,\II}(y,\xi)-S_{f,\II}(\overline{y},\xi)))\right)\,\mr{ds}_{\II}(\xi)=\\
=\,&\frac{1}{2}\left(\int_{\partial\Omega_{\II}}S_{f,\II}(y,\xi)\,\mr{ds}_{\II}(\xi)+\int_{\partial\Omega_{\II}}S_{f,\II}(\overline{y},\xi)\,\mr{ds}_{\II}(\xi)\right)+\\
&-\frac{I}{2}\left(J\left(\int_{\partial\Omega_{\II}}S_{f,\II}(y,\xi)\,\mr{ds}_{\II}(\xi)-\int_{\partial\Omega_{\II}}S_{f,\II}(\overline{y},\xi)\,\mr{ds}_{\II}(\xi)\right)\right)=\\
=\,&\frac{1}{2}(\mathsf{S}(y)+\mathsf{S}(\overline{y}))-\frac{I}{2}(J(\mathsf{S}(y)-\mathsf{S}(\overline{y}))).
\end{align*}
Implication $(3)\Longrightarrow(1)$ in Lemma \ref{lem:sliceness-criterion} ensures that $\mathsf{S}$ is a slice function. The same argument shows that $\mathsf{V}$ is a slice function.
\end{proof}

\begin{proof}[Proof of Theorem \ref{thm:Cauchy-int-f}]
Theorem 5.1 of \cite{JRS2019} asserts that, for each $x\in\Omega_{\II}$, it holds:
\[
f(x)=\int_{\partial\Omega_{\II}}\frac{1}{2\pi^2}\frac{\overline{\xi-x}}{|\xi-x|^4}(\n_{\II}(\xi)f(\xi))\,\mr{ds}_{\II}(\xi)-\int_{\Omega_{\II}}\frac{1}{2\pi^2}\frac{\overline{\xi-x}}{|\xi-x|^4}(\DD_{\II}f(\xi))\,\mr{dv}_{\II}(\xi).
\]
In particular, thanks to Lemma \ref{lem:char-Cauchy}, for each $x\in\Omega_D\cap\C_I\subset\Omega_{\II}$, we have:
\begin{equation*}
f(x)=\int_{\partial\Omega_{\II}}S_{f,\II}(x,\xi)\,\mr{ds}_{\II}(\xi)-\int_{\Omega_{\II}}V_{f,\II}(x,\xi)\,\mr{dv}_{\II}(\xi).
\end{equation*}
By Corollary \ref{cor:int-sliceness}, the function $\Omega_D\to\oo, x \mapsto\int_{\partial\Omega_{\II}}S_{f,\II}(x,\xi)\,\mr{ds}_{\II}(\xi)-\int_{\Omega_{\II}}V_{f,\II}(x,\xi)\,\mr{dv}_{\II}(\xi)$ is slice. Hence Corollary \ref{cor:ip} implies that the preceding equation holds for each $x\in\Omega_D$, i.e. \eqref{eq:BP} holds.

If $f$ is slice Fueter-regular, then \eqref{eq:sfrfDforall}, \eqref{eq:VfIxi} and \eqref{eq:VfI} imply that $V_{f,\II}=0$, so \eqref{eq:BP} reduces to \eqref{eq:C}. Moreover, equality (5.12) of Corollary 5.2 of \cite{JRS2019} ensures that $\int_{\partial\Omega_{\II}}S_{f,\II}(x,\xi)\,\mr{ds}_{\II}(\xi)=0$ for each $x\in\hh_{\II}\setminus\mr{cl}(\Omega_D)$. Repeating the preceding argument, we obtain that the same equality holds for each $x\in\oo\setminus\mr{cl}(\Omega_D)$. This proves \eqref{eq:C0}.
\end{proof}


\subsection*{Proofs of Sections \ref{subsec:taylor-expansions} and \ref{subsec:laurent-expansions}}

Fix $\II=(I,J)\in\mc{N}$. Let $\kappa\in\N^3$. Recall that $P_\kappa:\hh\to\hh\subset\oo$ is the standard $\kappa^{\mr{th}}$ Fueter polynomial, and $Q_\kappa:\hh\setminus\{0\}\to\hh\subset\oo$ the function $\partial_\kappa\Env$, where $\Env(x)=\frac{1}{2\pi^2}\frac{\overline{x}}{|x|^4}$. Furthermore, $\psi_{\II}:\hh\to\oo$ denotes the real algebra embedding sending $x_0+x_1i+x_2j+x_3k$ into $x_0+x_1I+x_2J+x_3IJ$. 

Let us give a characterization of $\Pp_{\II,\kappa}:\oo^2\to\oo$ and $\Qq_{\II,\kappa}:(\oo\setminus\{0\})\times\oo\to\oo$.

\begin{lem}\label{lem:ppqq}
$\Pp_{\II,\kappa}$ is the unique function $\mc{P}:\oo^2\to\oo$ satisfying the following condition: for each $c\in\oo$, the restriction function $\mc{P}|_c:\oo\to\oo$, defined by $\mc{P}|_c(x):=\mc{P}(x,c)$, is a slice function such that
\begin{equation}\label{eq:ffueter}
\mc{P}|_c(x)=\psi_{\II}(P_\kappa(\psi_{\II}^{-1}(x)))c \quad \text{ for each $x\in\C_I$.}
\end{equation}

Similarly, $\Qq_{\II,\kappa}$ is the unique function $\mc{Q}:(\oo\setminus\{0\})\times\oo\to\oo$ satisfying the following condition: for each $c\in\oo$, the restriction function $\mc{Q}|_c:\oo\setminus\{0\}\to\oo$, defined by $\mc{Q}|_c(x):=\mc{Q}(x,c)$, is a slice function such that
\begin{equation}\label{eq:fffueter}
\mc{Q}|_c(x)=\psi_{\II}(Q_\kappa(\psi_{\II}^{-1}(x)))c \quad \text{ for each $x\in\C_I\setminus\{0\}$.}
\end{equation}
\end{lem}
\begin{proof}
We follow the strategy used in the proof of Lemma \ref{lem:char-Cauchy}. We prove only the first part of the present lemma, the proof of the second being similar. Let $x=x_0+x_1I\in\C_I$ and let $\x\in\C_i$ with $\psi_{\II}(\x)=x$, i.e $\x=x_0+x_1i$. By Corollary \ref{cor:ip}, it suffices to show that $\Pp_{\II,\kappa}(x,c)=\psi_{\II}(P_\kappa(\x))c$. Bearing in mind \eqref{eq:lIIk}, \eqref{eq:I-product}, \eqref{eq:sfpf}, and using Artin's theorem, if $\kk=|\kappa|$, we have:
\begin{align*}
\Pp_{\II,\kappa}(x,c)&=\sum_{h=0}^{\lfloor \frac{\kk}{2}\rfloor} x_0^{\kk-2h}(x_1I)^{2h}(\ell_{\II,\kappa,2h}c)+\sum_{h=0}^{\lfloor \frac{\kk-1}{2}\rfloor}x_0^{\kk-2h-1}(x_1I)^{2h+1}(-I((I\ell_{\II,\kappa,2h+1})c))=\\
&=\sum_{h=0}^{\lfloor \frac{\kk}{2}\rfloor} x_0^{\kk-2h}(-x_1^2)^h(\ell_{\II,\kappa,2h}c)+\sum_{h=0}^{\lfloor \frac{\kk-1}{2}\rfloor}x_0^{\kk-2h-1}(-x_1^2)^h(x_1I)(-I((I\ell_{\II,\kappa,2h+1})c))=\\
&=\sum_{h=0}^{\lfloor \frac{\kk}{2}\rfloor} (x_0^{\kk-2h}(-x_1^2)^h\ell_{\II,\kappa,2h})c+\sum_{h=0}^{\lfloor \frac{\kk-1}{2}\rfloor}x_0^{\kk-2h-1}(-x_1^2)^hx_1(I\ell_{\II,\kappa,2h+1})c=\\
&=\sum_{h=0}^{\lfloor \frac{\kk}{2}\rfloor} (x_0^{\kk-2h}(-x_1^2)^h\ell_{\II,\kappa,2h})c+\sum_{h=0}^{\lfloor \frac{\kk-1}{2}\rfloor}(x_0^{\kk-2h-1}(x_1I)(-x_1^2)^h\ell_{\II,\kappa,2h+1})c=\\
&=\Pp_{\II,\kappa}(x,1)c=\psi_{\II}(P_\kappa(\x))c.
\end{align*}

The proof is complete.
\end{proof}

\begin{proof}[Proof of Theorem \ref{thm:taylor}]
Theorem 6.1 of \cite{JRS2019} asserts that, for each $x\in B_{\II}(r):=B(r)\cap\hh_\II$, it holds:
\begin{equation}\label{eq:k}
f(x)=\sum_{\kk\in\N}\sum_{\kappa\in\N^3,|\kappa|=\kk}\psi_\II(P_\kappa(\psi_{\II}^{-1}(x)))\frac{\partial_\kappa f(0)}{\kappa!},
\end{equation}
where the series converges uniformly on compact subsets of $B_\II(r)$. Actually, this series converges absolutely in the sense of \cite[Section 9.1.3]{GHS2008}, and hence it converges totally in the sense of \cite[Definition 3.1]{GPS2017-singular}. In particular, the same is true replacing $B_\II(r)$ with $B_I(r):=B_\II(r)\cap\C_I=B(r)\cap\C_I$. Combining the latter fact with Lemma \ref{lem:ppqq} and Definition \ref{def:sfpf}, we deduce:
\begin{equation*}
f(x)=\sum_{\kk\in\N}\sum_{\kappa\in\N^3,|\kappa|=\kk}\Pp_{\II,\kappa}(x,\textstyle\frac{\partial_\kappa f(0)}{\kappa!})=\sum_{\kk\in\N}\Pp_{\II,\kk;f,0}(x) \quad \text{ for each $x\in B_I(r)$}
\end{equation*}
and the series converges uniformly on compact subsets of $B_I(r)$. Now  \cite[Theorem 3.4]{GPS2017-singular} and above Corollary \ref{cor:ip} ensure that $f=\sum_{\kk\in\N}\Pp_{\II,\kk;f,0}$ on the whole $B(r)$, where the series $\sum_{\kk\in\N}\Pp_{\II,\kk;f,0}$ converges uniformly on compact subsets of $B(r)$.
\end{proof}

\begin{proof}[Proof of Corollary \ref{cor:taylor}]
Given any $n,m\in\N$ and $a\in\oo$, it is easy to verify that $\overline{\partial}(\mr{Re}(x)^n)=\frac{n}{2}\mr{Re}(x)^{n-1}$, $\overline{\partial}(\mr{Im}(x)^m)=-\frac{m}{2}\mr{Im}(x)^{m-1}$ and hence $\overline{\partial}(\mr{Re}(x)^n\mr{Im}(x)^ma)=\frac{n}{2}\mr{Re}(x)^{n-1}\mr{Im}(x)^ma-\frac{m}{2}\mr{Re}(x)^n\mr{Im}(x)^{m-1}a$. Furthermore, $(\mr{Re}(x)^n\mr{Im}(x)^ma)_s'=0$ if $m$ is even and $(\mr{Re}(x)^n\mr{Im}(x)^ma)_s'=\mr{Re}(x)^n\mr{Im}(x)^{m-1}a$ if $m$ is odd. By \eqref{eq:sfpf} and Definition \ref{def:sfpf}, it follows that the functions $\overline{\partial}\Pp_{\II,\kk;f.0}:B(r)\to\oo$ and $(\Pp_{\II,\kk;f.0})_s':B(r)\to\oo$ have the following property: each of their eight real components is either the zero function or a homogeneous polynomial function of degree $\kk-1$ in the eight real components of $x$. Thanks to Theorem \ref{thm:2}, we know that $\overline{\partial} f=f_s'$ on $B(r)\setminus\R$. Since the series $\sum_{\kk\in\N}\Pp_{\II,\kk;f,0}$ converges uniformly on compact subsets of $B(r)$, we deduce that $\sum_{\kk\in\N}\overline{\partial}\Pp_{\II,\kk;f,0}=\sum_{\kk\in\N}(\Pp_{\II,\kk;f,0})_s'$ on $B(r)\setminus\R$, where the latter two series converges uniformly on compact subsets of $B(r)\setminus\R$. The above homogeneity property of $\overline{\partial}\Pp_{\II,\kk;f.0}$ and $(\Pp_{\II,\kk;f.0})_s'$ implies that $\overline{\partial}\Pp_{\II,\kk;f,0}=(\Pp_{\II,\kk;f,0})_s'$ on $B(r)\setminus\R$ for each $\kk\in\N$. Hence each $\Pp_{\II,\kk;f,0}$ is a slice Fueter-regular function.
\end{proof}

\begin{proof}[Proof of Theorem \ref{thm:laurent} and Corollary \ref{cor:laurent}]
To prove these results, it is suffices to repeat the arguments used in the proofs of Theorem \ref{thm:taylor} and Corollary \ref{cor:taylor} respectively, replacing \cite[Theorem 6.1]{JRS2019} and \eqref{eq:sfpf} with \cite[Theorem 6.5]{JRS2019} and \eqref{eq:QqIIk}, respectively.
\end{proof}




\subsection*{Proofs of Section \ref{subsec:mmp}} To prove Theorem \ref{thm:MMP}, we need two preliminary results. First, we reformulate a lemma of \cite{GPS2019}.

\begin{lem}[Lemma 4.14 of {\cite{GPS2019}}]\label{lem:oo}
Let $f:\Omega_D\to\oo$ be a slice function and let $c$ be a constant in $\oo\setminus\{0\}$. Then the following holds:
\begin{itemize}
 \item[(1)] For each $x\in\Omega_D$, there exists $y_x\in\sph_x$ such that $(f\cdot c)(x)=f(y_x)c$.
 \item[(2)] If $f$ is continuous and $\Omega_D':=\Omega_D\setminus (\R\cup V(f_s'))\neq\emptyset$, then there exists a homeomorphism $\Phi:\Omega_D'\to\Omega_D'$ such that $\Phi(\sph_x)=\sph_x$ and $(f\cdot c)(x)=f(\Phi(x))c$ for each $x\in\Omega_D'$. Furthermore,
\begin{align}
\Phi(x)&=((x(f_s'(x)c))c^{-1})f_s'(x)^{-1},\label{eq:12}\\
\Phi^{-1}(x)&=((xf_s'(x))c)(c^{-1}f_s'(x)^{-1})\label{eq:13}
\end{align}
for each $x\in\Omega_D'$.
\end{itemize}
\end{lem} 
\begin{proof}
The mentioned lemma of \cite{GPS2019} implies at once assertion (1), and it ensures that the map $\Phi:\Omega_D'\to\Omega_D'$, defined by formula \eqref{eq:12}, has the following property: for each $x\in\Omega_D'$, $\Phi(x)$ is the unique element of $\sph_x$ such that $(f\cdot c)(x)=f(\Phi(x))c$. Moreover, if $f$ is continuous, then $f_s'$ and hence $\Phi$ are continuous as well. Note that $(f\cdot c)_s'=f_s'c$. Hence $V((f\cdot c)_s')=V(f_s')$ and, given any $x\in\Omega_D'$, it holds:
\[
((f\cdot c)\cdot c^{-1})(\Phi(x))=f(\Phi(x))=((f\cdot c)(x))c^{-1}.
\]
In this way, applying Lemma 4.14 of \cite{GPS2019} with $f\cdot c$ in place of $f$ and $c^{-1}$ in place of $c$, we obtain that $\Phi$ is bijective and its inverse has the following form:
\begin{align*}
\Phi^{-1}(x)&=((x(((f\cdot c)_s'(x))c^{-1}))c)((f\cdot c)_s'(x))^{-1}\\
&=((xf_s'(x))c)(c^{-1}f_s'(x)^{-1}).
\end{align*}
In particular, $\Phi^{-1}$ is continuous as well.
\end{proof}

Let $f:\Omega_D\to\oo$ be a slice function and let $x$ be a point of $\Omega_D$ such that $V(f)\cap\sph_x=\{x\}$, i.e. $x$ is the unique zero of $f$ on $\sph_x$. Given $c\in\oo\setminus\{0\}$, there exists a unique point $x_c\in\sph_x$ such that $V(f\cdot c)\cap\sph_x=\{x_c\}$. Usually, $x_c\neq x$. This phenomenon, called camshaft effect, is due to the non-associativity of $\oo$, see \cite{GP2011-camshaft,GPS2017,GPS2019} for the general version of such an effect.

The mentioned camshaft effect can be also interpreted as a measure of the difference between the pointwise and the slice products. Let $y\in\Omega_D$ and let $g:\Omega_D\to\oo$ be the slice function $g(x):=(f\cdot c)(x)-f(y)c=(f(x)-f(y))\cdot c$. As we said, there exists a unique $y_c\in\sph_y$ such that $g(y_c)=0$ or, equivalently, $(f\cdot c)(y_c)=f(y)c$. Usually, $y_c\neq y$. 

The second preliminary result we need is the following lemma, which pointwise `annihilates' the camshaft effect in the case of a slice product with a certain pair of constants on the right.

\begin{lem}\label{lem:F1-F2}
Let $f:\Omega_D\to\oo$ be a slice function induced by the stem function $(F_1,F_2):D\to\oo^2$, let $p=\alpha+I\beta\in\Omega_D$ with $\alpha,\beta\in\R$ and $I\in\sph$, and let $w:=\alpha+i\beta\in D$. Define $\gamma:=F_1(w)$, $\delta:=F_2(w)$, $q:=f(p)=\gamma+I\delta$ and the slice functions $g_1,g_2:\Omega_D\to\oo$ by setting
\[
g_1:=(f\cdot\overline{\gamma})\cdot((-\gamma(\overline{q}I))I)
\quad\text{\it and }\quad
g_2:=(f\cdot\overline{\delta})\cdot(\delta\overline{q}).
\]
 
Then $g_1(p)=|\gamma|^2|q|^2$ and $g_2(p)=|\delta|^2|q|^2$.
\end{lem}
\begin{proof}
Let $b:=(-\gamma(\overline{q}I))I$. Thanks to Artin's theorem, we have: 
\[
b=-(\gamma((\overline{\gamma}-\overline{\delta}I)I))I=-(\gamma(\overline{\gamma}I+\overline{\delta}))I=-(|\gamma|^2I+\gamma\overline{\delta})I=|\gamma|^2-(\gamma\overline{\delta})I.
\]
Let $(G_{1,1},G_{1,2}):D\to\oo^2$ be the stem function inducing $g_1$. Since $G_{1,1}(z)=(F_1(z)\overline{\gamma})b$ and $G_{1,2}(z)=(F_2(z)\overline{\gamma})b$ for each $z\in D$, setting $z=w$, we deduce:
\begin{align*}
g_1(p)&=(F_1(w)\overline{\gamma})b+I((F_2(w)\overline{\gamma})b)=(\gamma\overline{\gamma})(|\gamma|^2-(\gamma\overline{\delta})I)+I((\delta\overline{\gamma})(|\gamma|^2-(\gamma\overline{\delta})I))=\\
&=|\gamma|^4-|\gamma|^2(\gamma\overline{\delta})I+|\gamma|^2I(\delta\overline{\gamma})+|\gamma|^2|\delta|^2=|\gamma|^2(|\gamma|^2+I(\delta\overline{\gamma})-(\gamma\overline{\delta})I+|\delta|^2).
\end{align*}
Bearing in mind that the trace function $\mr{t}:\oo\to\R$, given by $\mr{t}(x):=x+\overline{x}$, vanishes on associators (see e.g. \cite[Lemma 5.6]{GPS2017}), we deduce:
\[
I(\delta\overline{\gamma})-(\gamma\overline{\delta})I=\mr{t}(I(\delta\overline{\gamma}))=\mr{t}((I\delta)\overline{\gamma})=(I\delta)\overline{\gamma}-\gamma(\overline{\delta}I)
\] 
and hence
\[
g_1(p)=|\gamma|^2(|\gamma|^2+(I\delta)\overline{\gamma}-\gamma(\overline{\delta}I)+|\delta|^2)=|\gamma|^2|\gamma+I\delta|^2=|\gamma|^2|q|^2.
\]

Denote $(G_{2,1},G_{2,2}):D\to\oo^2$ the stem function inducing $g_2$. Since $G_{2,1}(z)=(F_1(z)\overline{\delta})(\delta\overline{q})$ and $G_{2,2}(z)=(F_2(z)\overline{\delta})(\delta\overline{q})$ for each $z\in D$, setting $z=w$, we deduce:
\begin{align*}
g_2(p)&=(F_1(w)\overline{\delta})(\delta\overline{q})+I((F_2(w)\overline{\delta})(\delta\overline{q}))=(\gamma\overline{\delta})(\delta\overline{\gamma}-|\delta|^2I)+I((\delta\overline{\delta})(\delta\overline{\gamma}-|\delta|^2I))=\\
&=|\gamma|^2|\delta|^2-|\delta|^2(\gamma\overline{\delta})I+|\delta|^2I(\delta\overline{\gamma})+|\delta|^4=|\delta|^2(|\gamma|^2+I(\delta\overline{\gamma})-(\gamma\overline{\delta})I+|\delta|^2)=\\
&=|\delta|^2(|\gamma|^2+(I\delta)\overline{\gamma}-\gamma(\overline{\delta}I)+|\delta|^2)=|\delta|^2|\gamma+I\delta|^2=|\delta|^2|q|^2.
\end{align*}
The proof is complete.
\end{proof}

The reader notes that $(-\gamma(\overline{q}I))I=\overline{\langle I;q,\overline{\gamma}\rangle}$, because of above definition \eqref{eq:I-product}.

\begin{cor}\label{cor:F1-F2}
Let $f:\Omega_D\to\oo$ be a slice function, let $p\in\Omega_D$ and let $q:=f(p)$. If $q\neq0$, then there exist $a,b\in\oo\setminus\{0\}$ such that $|b|=|a||q|$ and the slice function $g:\Omega_D\to\oo$, defined by $g:=(f\cdot a)\cdot b$, has the following property:
\[
g(p)=|a|^2|q|^2>0.
\]

Furthermore, if $p\not\in\R\cup V(f_s')$, then we can set $a=\overline{f_s'(p)}$ and $b=(f_s'(p))\overline{q}$.
\end{cor}
\begin{proof}
Let $(F_1,F_2)$ be the stem function inducing $f$, let $p=\alpha+I\beta$ and let $w:=\alpha+i\beta$. Note that $0\neq q=F_1(w)+IF_2(w)$, so either $F_1(w)\neq0$ or $F_2(w)\neq0$. Thanks to the preceding lemma, if $F_1(w)\neq0$, then it suffices to set  $g:=(f\cdot a)\cdot b$ with $a:=\overline{F_1(w)}$ and $b:=-((F_1(w))(\overline{q}I))I\neq0$. If $F_2(w)\neq0$, we can set $a:=\frac{\overline{F_2(w)}}{\beta}=\overline{f_s'(p)}$ and $b:=\frac{(F_2(w))\overline{q}}{\beta}=(f_s'(p))\overline{q}\neq0$.
\end{proof}

\begin{proof}[Proof of Theorem \ref{thm:MMP}]
Let $p$ be a point of $\Omega_D$.

Let us prove point (1). 

Suppose that $|f|$ has a maximum at~$p$. Let $q:=f(p)$. If $q=0$, then $f=0$ and the proof is complete. Suppose $q\neq0$. By Corollary~\ref{cor:F1-F2}, there exist $a,b\in\oo\setminus\{0\}$ such that $|b|=|a||q|$ and the slice function $g:\Omega_D\to\oo$, defined by $g:=(f\cdot a)\cdot b$, has the following property: $g(p)=|a|^2|q|^2>0$. Moreover, if $p\not\in\R\cup V(f_s')$, then we can assume $a=\overline{f_s'(p)}$ and $b=(f_s'(p))\overline{q}$. By Corollary \ref{cor:4}, we also known that $g$ is slice Fueter-regular. Note that, if $(F_1,F_2)$ and $(G_1,G_2)$ are the stem functions inducing $f$ and $g$ respectively, then $G_h=(F_ha)b$ for $h\in\{1,2\}$. In this way, $f$ is constant, i.e. $F_1$ is constant and $F_2=0$, if and only if the same is true for $G_1$ and $G_2$, i.e. $g$ is constant. We will show that $g$ is constant completing the proof.

Applying twice Lemma \ref{lem:oo}(1) to $g$, we see that, given any $x\in\Omega_D$, there exist $y_x,w_x\in\sph_x\subset\Omega_D$ such that
\[
g(x)=((f\cdot a)(y_x))b=(f(w_x)a)b,
\]
so
\begin{equation}\label{eq:bound}
|g(x)|=|f(w_x)||a|^2|q|\leq|a|^2|q|^2=g(p)=|g(p)|.
\end{equation}

Choose $\II=(I,J)\in\mc{N}$ in such a way that $p\in\C_I$, and let $L\in\sph$ be such that $L\perp\hh_{\II}$. Denote $\hat{g}:\Omega_\II=\Omega_D\cap\hh_{\II}\to\oo$ the restriction of $g$ to $\hh_{\II}$. By Lemma 4.6 and Remark 4.7 of \cite{JRS2019}, there exist $\hat{g}_1,\hat{g}_2:\Omega_\II\to\hh_{\II}$ such that $\hat{g}_1$ is (standard left) Fueter-regular in $\hh_{\II}=\hh$, $\hat{g}_2$ is (standard) right Fueter-regular in $\hh_{\II}=\hh$ and $\hat{g}=\hat{g}_1+\hat{g}_2L$. Note that $\hat{g}_1(p)\in\R$, $\hat{g}_2(p)=0$ and $|\hat{g}(p)|=g(p)=\hat{g}_1(p)>0$. By \eqref{eq:bound}, we deduce that $p$ is a maximum of $|\hat{g}|$. Note also that $\Omega_\II$ is connected, because it is dense in $\Omega_\II\setminus\R$ and $\Omega_\II\setminus\R$ is homeomorphic to $D^+\times\sph^2$, where $D^+=\{(x_0,x_1)\in D\,:\,x_1>0\}$ is connected by assumption \eqref{assumption}.


Now we can adapt to the present situation the argument used in the proof of Theorem 7.1 of \cite{GSSbook}. It holds:
\begin{equation}
|\hat{g}_1(p)|^2=|\hat{g}(p)|^2\geq|\hat{g}(x)|^2=|\hat{g}_1(x)|^2+|\hat{g}_2(x)|^2\geq|\hat{g}_1(x)|^2
\end{equation}
for each $x\in\Omega_\II$. It follows that the modulus of the Fueter-regular function $\hat{g}_1:\Omega_\II\to\hh_{\II}$ has a maximum at a point of $\Omega_\II$. On the other hand, Fueter-regular functions enjoy Maximum Modulus Principle (see \cite[Theorem 7.32]{GHS2008}), so $\hat{g}_1$ is constant on its (connected) domain $\Omega_\II$. As a consequence, if $x\in\Omega_\II$, we have:
\[
|\hat{g}_2(x)|^2=|\hat{g}(x)|^2-|\hat{g}_1(x)|^2=|\hat{g}(x)|^2-|\hat{g}(p)|^2\leq0,
\]
hence $\hat{g}_2(x)=0$ and $\hat{g}(x)=\hat{g}_1(x)+\hat{g}_2(x)L=\hat{g}_1(p)=\hat{g}(p)$. In particular, $\hat{g}$ (and hence $g$) is constantly equal to $\hat{g}(p)=g(p)$ on $\C_I\cap\Omega_D$. Thanks to Corollary \ref{cor:ip}, it follows that $g=g(p)$ on $\Omega_D\setminus\R$, and hence on the whole $\Omega_D$ by continuity.

It remains to show point (2). Suppose that $p\not\in V(f_s')$ and $|f|$ has a local maximum at $p$. If $q=f(p)=0$, then $f=0$ locally at $p$ in $\Omega_D$. Since $\Omega_D$ is connected and $f$ is real analytic, it follows that $f=0$ on the whole $\Omega_D$, and we are done. 

Let $q\neq0$. We distinguish two cases.

First, we assume that $p\not\in\R$, so $p\in\Omega_D':=\Omega_D\setminus(\R\cup V(f_s'))$. As above we define $a:=\overline{f_s'(p)}$, $b:=(f_s'(p))\overline{q}=\overline{qa}$ and $g:=(f\cdot a)\cdot b$. Let $U$ be a neighborhood of $p$ in $\Omega_D$ such that $|f(x)|\leq|f(p)|$ for each $x\in U$. Since $\Omega_D'$ is open in $\Omega_D$ and contains $x$, we can also assume that $U\subset\Omega_D'$. Apply twice Lemma \ref{lem:oo}(2) to $g$. We obtain the existence of a homeomorphism $\Psi:\Omega_D'\to\Omega_D'$ such that, for each $x\in\Omega_D'$, $\Psi(\sph_x)=\sph_x$ and $g(x)=(f(\Psi(x))a)b$.

Let us prove that $\Psi(p)=p$. Since we have
\[
(qa)(\overline{qa})=|q|^2|a|^2=g(p)=(f(\Psi(p))a)b=(f(\Psi(p))a)(\overline{qa}),
\]
dividing first by $\overline{qa}\neq0$ and then by $a\neq0$ on the right, we deduce that $f(p)=q=f(\Psi(p))$. On the other hand, $f$ is slice, so $f|_{\sph_p}(x)=c+x(f_s'(p))$ for some $c\in\oo$. Since $f_s'(p)\neq0$, $f|_{\sph_p}$ turns out to be injective. This proves that $\Psi(p)=p$, as claimed.

Now the continuity of $\Psi$ implies the existence of a neighborhood $V$ of $p$ in $\Omega_D'$ such that $\Psi(V)\subset U$. In this way, if $x\in V$, then we have:
\[
|g(x)|=|f(\Psi(x))||a|^2|q|\leq|f(p)||a|^2|q|=|a|^2|q|^2=|g(p)|.
\]
Using $\II$, $L$ and $\hat{g}=\hat{g}_1+\hat{g}_2L$ as above, we conclude that $g$ is constantly equal to $g(p)$ on the whole $\Omega_D$, completing the proof.

Finally, we assume that $p\in\R\cap\Omega_D$. Applying twice Lemma \ref{lem:oo}(1) to $g$ again, we obtain that $p$ is a local maximum for $|g|$. Hence, considering $\II$, $L$ and $\hat{g}=\hat{g}_1+\hat{g}_2L$ as above, we deduce that $g=g(p)$ locally at $p$ in $\Omega_D$. Since $\Omega_D$ is connected and $g$ is real analytic, $g$ turns out to be constantly equal to $g(p)$ on the whole $\Omega_D$.
\end{proof}



\section{Final remarks}\label{sec:final}

We conclude the present paper with some comments concerning possible generalizations of the concept of slice Fueter-regular function we studied above. 

Let $A$ be an arbitrary finite dimensional real alternative $^*$-algebra with unity, equipped with its natural smooth manifold structure as a real vector space. Consider again a non-empty open subset $D$ of $\C$ invariant under the complex conjugation. Let $F=(F_1,F_2):D\to A^2$ be a stem function of class $\mscr{C}^1$. Denote $\Omega_D$ the circularization of $D$ in $A$ and $f=\I(F):\Omega_D\to A$ the slice function induced by $F$, i.e.
\[
f(x_0+x_1 I):=F_1(x_0,x_1)+IF_2(x_0,x_1)
\]
for each $(x_0,x_1)\in D$ and $I\in\sph_A:=\{x\in A\,:\,x^*=-x,x^2=-1\}$. Recall that $\Omega_D$ is an open subset of the quadratic cone $Q_A$ of $A$. We refer the reader to \cite{GP2011} for details. 


\subsection{Slice Vekua-regular functions}\label{subsec:svrf}

Let $\mr{V}=(a,b,c,d):D\setminus\R\to A^4$ be a function {\it even-odd-odd-even in $x_1$}, i.e. $a(x_0,-x_1)=a(x_0,x_1)$, $b(x_0,-x_1)=-b(x_0,x_1)$, $c(x_0,-x_1)=-c(x_0,x_1)$ and $d(x_0,-x_1)=d(x_0,x_1)$.

Consider the Vekua-type system $\mathsf{S}_\mr{V}$ induced by $\mr{V}$:
\begin{equation}\label{eq:V}
\mathsf{S}_\mr{V}:\;
\left\{
\begin{array}{ll}
\displaystyle
\frac{\partial F_1}{\partial x_0}-\frac{\partial F_2}{\partial x_1}+aF_1+bF_2=0
\vspace{.7em}\\
\displaystyle
\frac{\partial F_1}{\partial x_1}+\frac{\partial F_2}{\partial x_0}+cF_1+dF_2=0\,.
\end{array}
\right.
\end{equation}

Note that the function $\frac{\partial F_1}{\partial x_0}-\frac{\partial F_2}{\partial x_1}+aF_1+bF_2$ is even in $x_1$ and the function $\frac{\partial F_1}{\partial x_1}+\frac{\partial F_2}{\partial x_0}+cF_1+dF_2$ is odd in $x_1$. If $\mr{V}$ is $\R^4$-valued, then above system $\mathsf{S}_\mr{V}$ is a particular case of systems known as {\it Bers-Vekua} or {\it Carleman-Bers-Vekua systems}. The study of the latter systems is the core of the pseudoanalytic function theory, see Sections 6.2 and 6.3 of \cite{GHS2016} and their references. 

A natural possibility to define the concept of slice Vekua-regular function is as follows:

\begin{defn}
Let $\mr{V}=(a,b,c,d):D\setminus\R\to A^4$ be any given even-odd-odd-even function in~$x_1$. We say that the slice function $f=\I((F_1,F_2)):\Omega_D\to A$ is a \emph{slice Vekua-regular function with respect to $\mr{V}$}, or a {\it slice $\mr{V}$-Vekua-regular function} for short, if the stem function $(F_1,F_2)$ inducing $f$ satisfies system $\mathsf{S}_\mr{V}$ defined in \eqref{eq:V}. \bs
\end{defn}

One can ask whether there exist pairs $(A,\mr{V})$ such that the corresponding slice $\mr{V}$-Vekua-regular functions $f$ share significant results with quasianalytic functions; or better, such that such functions $f$ 
have a `holomorphic nature', satisfying for example versions of Cauchy integral formula, Taylor and Laurent series expansions, and Maximum Modulus Principle.

The results of the present paper show that the answer is yes, at least in the case $A:=\oo$ and $\mr{V}(x_0,x_1):=(0,2x_1^{-1},0,0)$.

\begin{quest}\label{quest:50}
Do there exist other relevant pairs $(A,\mr{V})$? 
\end{quest}


\subsection{Slice Dirac-regular functions}\label{subsec:sdrf}

A Dirac operator can be defined as a first order differential operator which factorizes a Laplacian. A remarkable example of Dirac operator is the Cauchy-Riemann-Fueter operator $\DD=\frac{\partial}{\partial x_0}+i\frac{\partial}{\partial x_1}+j\frac{\partial}{\partial x_2}+k\frac{\partial}{\partial x_3}$ on quaternions. Similar examples can be defined over other real alternative $^*$-algebras as Clifford algebras and octonions.

In this way, one can generalize the concept of slice Fueter-regular function as follows. Let $A$ be a finite dimensional real alternative $^*$-algebra with unity, let $B$ be a real subalgebra of $A$ and let $\mscr{D}$ be a Dirac operator on $B$. Given a slice function $f:\Omega_D\to A$ of class $\mscr{C}^1$, we say that $f$ is a {\it slice $\mscr{D}$-Dirac-regular function} if the restriction of $f$ to $B\cap\Omega_D$ belongs to the kernel of $\mscr{D}$. If the triple $(A,B,\mscr{D})$ is equal to $(\oo,\hh,\DD)$, then above assertion \eqref{eq:sfrfDexists} implies that the just defined concept of slice $\mscr{D}$-Dirac-regular function coincides with the one of slice Fueter-regular function.  

Also in this case, it is natural to ask the following:

\begin{quest}\label{quest:51}
Do there exist other relevant triples $(A,B,\mscr{D})$?
\end{quest}



\end{document}